\documentclass[11pt]{amsart}

\usepackage[T1]{fontenc}
\usepackage[utf8]{inputenc}
\usepackage{textcomp}
\usepackage{lmodern}
\usepackage{microtype}

\usepackage{amssymb}
\usepackage{amsthm}
\usepackage{amscd}
\usepackage{amstext}
\usepackage{enumitem}

\usepackage{hyperref}

\makeatletter
\newtheorem*{rep@theorem}{\rep@title}
\newcommand{\newreptheorem}[2]{%
\newenvironment{rep#1}[1]{%
 \def\rep@title{#2 \ref{##1}}%
 \begin{rep@theorem}}%
 {\end{rep@theorem}}}
\makeatother

\usepackage{mathscinet}
\usepackage[giveninits=true,doi=false,url=false,sortcites,backend=biber,backref=true]{biblatex}
\newtheorem{theorem}{Theorem}[section]
\newreptheorem{theorem}{Theorem}

\newtheorem{lemma}[theorem]{Lemma}

\newtheorem{cor}[theorem]{Corollary}
\newreptheorem{cor}{Corollary}

\theoremstyle{definition}
\newtheorem{definition}[theorem]{Definition}

 \def \dom{\operatorname{dom}}

 \allowdisplaybreaks[2]

\begin{document}

\title{A Removal Lemma for Ordered Hypergraphs}
\author{Henry Towsner}
\date{\today}
\thanks{Partially supported by NSF grant DMS-1600263}
\address {Department of Mathematics, University of Pennsylvania, 209 South 33rd Street, Philadelphia, PA 19104-6395, USA}
\email{htowsner@math.upenn.edu}
\urladdr{\url{http://www.math.upenn.edu/~htowsner}}

\begin{abstract}
  We prove a removal lemma for induced ordered hypergraphs, simultaneously generalizing Alon--Ben-Eliezer--Fischer's removal lemma for ordered graphs and the induced hypergraph removal lemma.  That is, we show that if an ordered hypergraph $(V,G,<)$ has few induced copies of a small ordered hypergraph $(W,H,\prec)$ then there is a small modification $G'$ so that $(V,G',<)$ has no induced copies of $(W,H,\prec)$.  (Note that we do \emph{not} need to modify the ordering $<$.)

  We give our proof in the setting of an ultraproduct (that is, a Keisler graded probability space), where we can give an abstract formulation of hypergraph removal in terms of sequences of $\sigma$-algebras.  We then show that ordered hypergraphs can be viewed as hypergraphs where we view the intervals as an additional notion of a ``very structured'' set.  Along the way we give an explicit construction of the bijection between the ultraproduct limit object and the corresponding hyerpgraphon.
\end{abstract}

\maketitle

\section{Introduction}

In this paper, we will show a removal lemma for ordered hypergraphs---a simultaneous generalization of the removal lemma for ordered graphs \cite{MR3734286,MR4102484} and for hypergraphs \cite{MR2373376,rodl:MR2069663,nagle:MR2198495}. 

As in similar results, the methods naturally generalize to finite colorings of $k$-tuples (``hypermatrices over a finite alphabet'').  Therefore, in full generality, our main result is the following.
\begin{repcor}{thm:main_ordered}
  Let $\epsilon>0$ be given and let $\Sigma$ be a finite alphabet.  There is a $\delta>0$ so that whenever $(\Omega,<)$ is an ordered set and $\rho:{\Omega\choose k}\rightarrow\Sigma$, there is a $\rho':{\Omega\choose k}\rightarrow\Sigma$ such that:
  \begin{itemize}
  \item $|\{\vec x\in{\Omega\choose k}\mid \rho(\vec x)\neq \rho'(\vec x)\}|<\epsilon|\Omega|^k$, and
  \item for each ordered set $(W,\prec)$ with $|W|<1/\epsilon$ and each coloring $c:{W\choose k}\rightarrow\Sigma$, either:
    \begin{itemize}
    \item $(\Omega,\rho',<)$ contains no copies of $(W,c,\prec)$ (that is, there are no order-preserving functions $\pi:W\rightarrow\Omega$ such that $\rho'(\pi(\vec x))=c(\vec x)$ for all $\vec x\in{W\choose k}$), or
    \item $(\Omega,\rho,<)$ contains many copies of $(W,c,\prec)$ (that is, the set of order-preserving functions $\pi:W\rightarrow\Omega$ such that $\rho(\pi(\vec x))=c(\vec x)$ for all $\vec x\in{W\choose k}$ has size at least $\delta|\Omega|^{|W|}$).
    \end{itemize}
  \end{itemize}
\end{repcor}

Coregliano and Razborov have recently \cite{MR4153699} shown a result that could also plausibly be called ordered hypergraph removal. Their removal involves modifying the entire structure---that is, one replaces $(\Omega,<,\rho)$ with $(\Omega,<',\rho')$, whereas the result here only modifies $\rho$. Their argument is quite general, applying to a wide range of structures. By contrast, our result is narrower, though we discuss at the end how the arguments might be generalized.

Our approach is to restate the usual proof of hypergraph removal in a sufficiently general way that the proof of ordered hypergraph removal falls out without much change.  We will consider $k$-graphs which have a sequence of notions of ``structured sets''.  In the usual graph removal lemma, this sequence would has length $1$: the only kind of structured set is the rectangles (that is, sets of edges of the form $A\times B$ for sets $A$ and $B$).

In the hypergraph removal lemma for $k$-graphs, the sequence of notions of structure has length $k-1$: the first, most general tier of structured sets are \emph{cylinder sets} generated by $k-1$-tuples (see the next section for a definition), then the next tier is the cylinder sets generate by $k-1$-tuples, and so on, until the final, most restrictive tier of structure consists of boxes (sets of the form $\prod_{i\leq k}A_i$, which are exactly the cylinder sets generated by $1$-tuples).

Meanwhile, in the ordered graph case, the sequence of notions of structure has length $2$: the more general tier of structured set is the rectangles $A\times B$ where $A,B$ are arbitrary sets of vertices, while the second, more restrictive tier of structure is sets of the form $I\times J$ where $I,J$ must be intervals in the ordering.

Once we have set up this general framework, ordered hypergraph removal will fall out almost instantly from the proof of hypergraph removal: we will use a sequence of notions of structure of length $k$, beginning with the cylinder sets generated by $k-1$-tuples, proceeding down to boxes of the form $\prod_{i\leq k}A_i$ where the $A_i$ are arbitrary, and then adding an additional notion of structured set given by boxes of the form $\prod_{i\leq k}I_i$ where the $I_i$ are intervals.

The idea that Szemer\'edi's regularity lemma and its generalizations can be viewed in terms of nested notions of structure is present, for instance, in \cite{MR2212136,MR2259060}, which describe Szemer\'edi's regularity lemma in terms of conditional expectation.

Working with multiple layers of structure typical requires fairly complicated dependencies to correctly express bounds in the finite setting, so it is convenient to pass to an infinitary, analytic setting where we can ``let $\epsilon$ equal $0$''---that is, where some of the bounds will disappear into a measure-theoretic limit object.

There are two main approaches to representing the notion of structure in such a formalism.  To be explicit, consider the case of a $3$-graph; in the finite setting, we have a large vertex set $V$ and consider some symmetric set $H\subseteq V^3$.  In one approach to limit objects, the graphon approach (e.g. \cite{lovasz:MR2274085}), the limit object is an uncountable space $\Lambda$ and a measurable function $f:\Lambda^6\rightarrow[0,1]$.  The first three coordinates correspond to the three coordinates in $V^3$, while the additional three coordinates correspond to the \emph{pairs} of coordinates.  More generally, if we began with sets $H\subseteq V^k$, the limit object would involve functions on $\Lambda^{2^k-2}$.  (The value $2^k-2$ corresponds to the subsets of $\{1,2,\ldots,k\}$ except for the empty set\footnote{In a third, related, setting---\emph{arrays of exchangeable random variables} \cite{MR2463439,aldous:MR637937,hoover:arrays}---the coordinate corresponding to the empty set is typically included, but easily eliminated because an exchangeable array of random variables is a combination of \emph{dissociated} arrays, where the dissociated arrays are precisely those where the coordinate corresponding to the empty set can be ignored.  A natural generalization of a $k$-hypergraphon would be to add the coordinate corresponding to the empty set; this would be roughly represent an ensemble of $k$-hypergraphons rather than a single such object.} and the whole set, which represents a ``purely quasirandom'' component which is omitted from the limit object\footnote{The coordinate corresponding to the full set is related to why we end up with a function rather than a set: we could work with a set $F\subseteq\Lambda^{2^k-1}$, and think of $f(\vec\omega)=\int f(\vec\omega,u)\,du$ where $u$ is the extra coordinate.})

Analogously, an ordered graph is a symmetric $H\subseteq V^2$ where $V$ is an ordered set.  An \emph{orderon} \cite{1811.02023} (the graphon-like limit object corresponding to an ordered graph) is then a function from $\Lambda^4$ to $[0,1]$: elements of $V$ are analogous to pairs from $\Lambda$, where the first component represents the information about the ordering and the second component the additional information which is present in the vertex which not explained by its position in the ordering. (Permutons \cite{MR2995721} and latinons \cite{2010.07854}---limit objects for permutations and Latin squares, respectively---similarly acquire ``extra'' coordinates in this way.)

It is both an advantage and a disadvantage of this representation that it fully separates out the interactions between these coordinates: the coordinates are combined as a familiar product measure space (it is common to take $\Lambda=[0,1]$, so the functions in question are simply measurable functions on $[0,1]^n$ for some $n$) and one can use standard results (for instance, the Lebesgue density theorem, as in \cite{MR2964622}) on the space.  However, because of this de-association of the coordinates, it is difficult to interpret what the higher order coordinates ``mean'' in a general way: when we represent a $3$-graph with a function $f(x,y,z,u,v,w)$, it is difficult to concretely say, in a general way, what a particular value of $u$ means.  (Indeed, it is artificial, and a bit misleading, to represent these objects as powers of a single space: for instance, there is no reason to think we can swap the $z$ and $u$ coordinates in a meaningful way.)

Here we prefer a different approach to the limit object where the limit objects have a more familiar form: our version of a limit of $3$-graphs will be a subset of $\Omega^3$ for an uncountable set $\Omega$, and our version of an ordered graph will be a subset of $\Omega^2$ where $\Omega$ is an ordered set.  The price is that we must work with a \emph{Keisler graded probability space}\footnote{See Section \ref{sec:keisler}.}.  This means that the measurable sets are more complicated: in addition to the sets of pairs given by the standard product measure construction, there are typically additional measurable sets of pairs which include the quasirandom sets.

In this setting, different kinds of structure are identified by looking at sub-$\sigma$-algebras of measurable sets, which represent notions of structure \cite{MR3583029, henry12:_analy_approac_spars_hyper}.  For instance, when we consider a $3$-graph $H\subseteq\Omega^3$, we have a collection $\mathcal{B}_2$ of measurable subsets of $\Omega^2$, containing all the information we need about the first two coordinates, but there is also a product $\sigma$-algebra, $\mathcal{B}_{2,1}\subseteq\mathcal{B}_2$, which is the collection of measurable sets generated by rectangles.  So the first two coordinates of $\Lambda^6$ correspond to determining information about sets in $\mathcal{B}_{2,1}$, while the fourth coordinate of $\Lambda^6$ corresponds to information about the quasirandom elements of $\mathcal{B}_2$.

Formally, the two approaches are linked by a suitable map $\pi:\Omega^3\rightarrow\Lambda^6$ where, for instance, $\pi^{-1}(C\times\Lambda^4)$ must give a set of the form $A\times\Omega$ where $A$ is $\mathcal{B}_{2,1}$-measurable, while $\pi^{-1}(\Lambda^3\times B\times\Lambda^2)$ must give a set of the form $A\times\Omega$ where $A$ is quasirandom.  (In fact, our approach to the proof will lead us to construct something close to an explicit version of this map.)

Similarly, when we have an ordering on $\Omega$ we have both the collection $\mathcal{B}_1$ of all measurable subsets of $\Omega$, but also a sub-$\sigma$-algebra $\mathcal{B}_<$ which is generated by the intervals. \footnote{The association between these approaches are quite systematic. For instance, the fact that orderons are functions with domain $\Lambda^4$ is essentially telling us that when we look at the Keisler graded probability space, we should be paying attention to a particular sub-$\sigma$-algebra of $\mathcal{B}_1$, namely $\mathcal{B}_<$.}

In this setting, we will be able to prove:
\begin{reptheorem}{thm:ordered_hypergraph_removal}
  Let $\Sigma$ be a finite alphabet and let $(\Omega,<)$ be a set together with a Keisler graded probability space on $\Omega$ such that $<$ is measurable, and suppose $\rho:{\Omega\choose k}\rightarrow\Sigma$ is measurable.  For each $w$ and each $\epsilon>0$, there is a $\rho':{\Omega\choose k}\rightarrow\Sigma$ such that
  \begin{itemize}
  \item $\mu(\{\vec x\mid\rho(\vec x)\neq \rho'(\vec x)\})<\epsilon$, and
  \item for each ordered set $(W,\prec)$ and coloring $c:{W\choose k}\rightarrow\Sigma$ , either:
    \begin{itemize}
    \item $(\Omega,\rho',<)$ contains no copies of $(W,c,\prec)$ (that is, there are no order-preserving functions $\pi:W\rightarrow\Omega$ such that $\rho'(\pi(\vec x))=c(\vec x)$ for all $\vec x\in{W\choose k}$), or
    \item $(\Omega,\rho,<)$ contains many copies of $(W,c,\prec)$ (that is, the set of order-preserving functions $\pi:W\rightarrow\Omega$ such that $\rho(\pi(\vec x))=c(\vec x)$ for all $\vec x\in{W\choose k}$ has positive measure).
    \end{itemize}
  \end{itemize}
\end{reptheorem}
Corollary \ref{thm:main_ordered} follows immediately by a standard ultraproduct argument as described in \cite{goldbring:_approx_logic_measure}.

The main technique in our proof will be reproducing, in this setting, something like the Lebesgue density theorem: a way of defining a notion of density in this setting so that almost every point is dense.

The use of infinitary and measure-theoretic arguments in proofs like this is superficial: rather than interpreting the proof of Theorem \ref{thm:ordered_hypergraph_removal} as actually involving infinite sets, one can interpret these as abbreviations for complicated statements about finite sets.  In particular, although the proof does not explicitly give bounds on the relationship between $\delta$ and $\epsilon$ in Theorem \ref{thm:ordered_hypergraph_removal}, it is routine (though quite tedious) to translate the proof into an explicit combinatorial one with bounds.  (See \cite{MR2964881,1804.10809} for a formal descriptions of how this may be done in general.)

The proof here uses a familiar structure, so we can already tell roughly what upper bound it gets: the ``unwound'' proof of Theorem \ref{thm:ordered_hypergraph_removal} for ordered $k$-graphs goes through a result similar to the regularity lemma for $k+1$-graphs.  This means the upper bound for removal of ordered $k$-graphs would be on the order of the $k+2$-th function in the fast-growing hierarchy.  (Recall that the second function in this family if roughly exponential, and later functions are obtained by iterating the previous function, so the $3$-rd function is iterated exponentiaion, the $4$-th function is the ``wowzer'' function obtained by iterating the iterated exponential, and so on.)  For hypergraph regularity, these bounds are known \cite{MR4025519} to be tight.  (For graph removal, better bounds \cite{fox:MR2811609} can be obtained by avoiding the regularity lemma.  It seems likely that similar methods can produce at least some improvement on the upper bound for hypergraph and ordered hypergraph removal.)

\section{Preliminaries}

\subsection{Cylinder Intersection Sets}

One of the central ideas in the proof of graph removal is approximating a graph $(V,E)$ using sets of vertices.  In modern presentations, this idea is usually expressed using some form of the Szemer\'edi regularity lemma: we find a partition $V=\bigcup_{i\leq n}V_i$ so that most of the bipartite graphs $(V_i,V_j,E\cap(V_i\times V_j))$ have a quasi-randomness property.  To turn this into a proof of graph removal, one of the key points is that almost all of the edges in $E$ will belong to bipartite graphs $(V_i,V_j,E\cap(V_i\times V_j))$ where the quasi-randomness property holds and the density $\frac{|E\cap(V_i\times V_j)|}{|V_i\times V_j|}\geq\epsilon>0$ is bounded away from $0$.

Letting $P$ be the set of pairs $(i,j)$ where this density is bounded away from $0$, we end up considering a related graph, $E^+=E\cap(\prod_{i,j\in P}V_i\times V_j)$, the edges which are near many other edges.

In order to generalize this to $k$-graphs, we need to generalize the idea of a product set to higher arity.  The right notion is a \emph{cylinder intersection set}: a cylinder intersection set of $k$-tuples is a collection of $k$-tuples defined by restricting the sets that certain $r$-sub-tuples can belong to for $r<k$.  For instance, if $k=3$ and $r=2$, the prototypical cylinder intersection set is a set of the form
\[\{(x,y,z)\mid (x,y)\in A,\ (x,z)\in B\, \text{ and }(y,z)\in C\}.\]
We can see that a product is just a cylinder intersection set where we only consider sub-tuples with $r=1$.

Since we will be considering cylinder intersection sets extensively, and since it turns out that we can view graph homomorphisms as themselves being cylinder intersection sets, it will be convenient to introduce some uniform notation.

We are interested in a situation where we have a finite set of points---say, $W$---with some structure (a $k$-graph or an ordering) and are interested in ``copies'' inside some other set $V$.  For this purpose, it is useful to work with ``$W$-tuples''.

\begin{definition}
  When $W$ is a finite set, a \emph{$W$-tuple} from $\Omega$ is a function $W$ to $\Omega$.  We write $\Omega^W$ for the set of $W$ tuples.

  When $k$ is a non-negative integer, we write $[k]$ for the set $\{1,2,\ldots,k\}$.
\end{definition}
These definitions equate a $[k]$-tuple with a $k$-tuple and a $\Omega^{[k]}$ with $\Omega^k$, so we can view this as an extension of the usual notation for tuples.
\begin{definition}
  When $\vec x_W\in \Omega^W$ is a $W$-tuple and $e\subseteq W$, we write $\vec x_e\in \Omega^e$ for the $e$-tuple $\vec x_e=\vec x_W\upharpoonright e$.
\end{definition}
When $e=\{i\}$ is a singleton, we can abbreviate $x_i=\vec x_{\{i\}}=\vec x_W(i)$ to recover the usual notation for tuples.

\begin{definition}
We write ${W\choose k}$ for the collection of subsets of $W$ of size $k$.  We write ${W\choose \leq k}$ for $\bigcup_{i\leq k}{W\choose i}$, the collection of subsets of $W$ of size $\leq k$, and ${W\choose <k}$ for $\bigcup_{i<k}{W\choose i}$.
  
Suppose that $W$ is a finite set and $S$ is a collection of subsets of $W$.  A \emph{$(W,S)$-cylinder intersection set} is a set of the form
\[T_{S}(\{A_e\}_{e\in S})=\{\vec x_W\mid\forall e\in S\ \vec x_e\in A_e\}\]
where each $A_e\subseteq W^e$.  We call the sets $A_e$ the \emph{components} of $T_{S}(\{A_e\}_{e\in S})$.
\end{definition}

We are particularly interested in two special cases of this definition.
\begin{definition}
  When $H=(W,F)$ is a finite $k$-graph and $G=(\Omega,E)$ is a $k$-graph, we define the \emph{copies of $H$ in $G$}, written $T_H(G)$, to be $T_F(\{E\})$.
\end{definition}
That is, the copies of $(W,F)$ in $(\Omega,E)$ are the tuples $\vec x_W$ such that, for every $e\in F$, $\vec x_e\in E$.

\begin{definition}
    When $H=(W,F)$ is a finite $k$-graph and $E\subseteq\Omega^k$ is a $k$-graph, we define the \emph{induced copies of $H$ in $G$}, written $T_H^{ind}(G)$ to be $T_{{W\choose k}}(\{A_e\})$ where $A_e=\left\{\begin{array}{ll}E&\text{if }e\in F\\\Omega^k\setminus E&\text{otherwise}\end{array}\right..$
  \end{definition}
That is, the induced copies of $(W,F)$ in $(\Omega,E)$ are the tuples $\vec x_W$ such that, for each $e\in{W\choose k}$, $e\in F$ if and only if $\vec x_e\in E$.

A few other kinds of cylinder intersection sets will be needed along the way.  Another case we will see is when $S={[k]\choose \leq k}$ or $S={[k]\choose <k}$.  We might think of these as ``simplices''.  For instance, when $S={[3]\choose < 3}$, $T_S(\{A_e\})$ is a set of triangles $(x_1,x_2,x_3)$ such that each vertex $x_i$ belongs to a set of vertices $A_{\{i\}}$ while each edge $(x_i,x_j)$ belongs to a set of edges $A_{\{i,j\}}$.

Results for induced graphs extend immediately to $\Sigma$-colorings where $\Sigma$ is some finite alphabet.
\begin{definition}
  Suppose $\rho:{\Omega\choose k}\rightarrow\Sigma$ and $c:{W\choose k}\rightarrow\Sigma$ are colorings.  The \emph{copies of $(W,c)$ in $(\Omega,\rho)$}, written $T_{W,c}(\Omega,\rho)$, are $T_{{W\choose k}}(\{A_e\})$ where $A_e=\{\vec x_e\mid \rho(\vec x_e)=c(e)\}$.
\end{definition}
An induced copy of $(W,F)$ in $(\Omega,E)$ is precisely a copy of $(W,\chi_F)$ in $(\Omega,\chi_E)$ where the characteristic functions $\chi_F,\chi_E$ are viewed as colorings where $\Sigma=\{0,1\}$.

\subsection{Measure Spaces}\label{sec:keisler}

It will be convenient for us to prove our results in an infinitary setting where we can use some measure theoretic ideas.  A \emph{Keisler graded probability space} consists of a set $\Omega$ and, for each $k$, a measure $\mu_k$ on subsets of $\Omega^k$.

When $\Omega$ is finite, the natural choice is to take each $\mu_k$ to be the counting measure, $\mu_k(S)=\frac{|S|}{|\Omega|^k}$, on subsets of $\Omega^k$.  When $\Omega$ is infinite, we need to fix $\sigma$-algebras of measurable sets and add some conditions to ensure that the measures are compatible with each other.

\begin{definition}
  A \emph{Keisler graded probability space} on $\Omega$ is a collection of probability measure spaces, $(\Omega^k,\mathcal{B}_k,\mu_k)$, for each $k\in\mathbb{N}$ so that:
  \begin{itemize}
  \item whenever $\pi:[1,k]\rightarrow[1,k]$ is a permutation and $B\in\mathcal{B}_k$, we have $B^\pi=\{(x_{\pi(1)},\ldots,x_{\pi(k)})\mid (x_1,\ldots,x_k)\in B\}\in\mathcal{B}_k$ and $\mu_k(B^\pi)=\mu_k(B)$,
  \item if $B\in\mathcal{B}_k$ and $C\in\mathcal{B}_r$ then $B\times C\in\mathcal{B}_{k+r}$,
  \item whenever $B\in\mathcal{B}_{k+r}$, the set of $(x_1,\ldots,x_r)$ such that $B_{x_1,\ldots,x_r}=\{(x_{r+1},\ldots,x_{k+r})\mid (x_1,\ldots,x_{k+1})\in B\}\in\mathcal{B}_k$ is a set in $\mathcal{B}_{r}$ of measure $1$ and
\[\mu_{k+r}(B)=\int\mu_k(B_{x_1,\ldots,x_r})\,d\mu_r.\]
\end{itemize}

We say $\{(\Omega^k,\mathcal{B}_k,\mu_k)\}_{k\in\mathbb{N}}$ is \emph{atomless} if, for every $x\in V$, $\mu_1(\{x\})=0$.
\end{definition}
Atomless Keisler graded probability spaces are the setting obtained by taking the limit of the counting measures as the size of $\Omega$ approaches infinity (made precise by using an ultraproduct).  As a result, for many purposes one can simply pretend that an atomless Keisler graded probability space is finite with $|\Omega|$ very, very large.

The special case where, for each $k$, $\mathcal{B}^k$ is equal to the product $\sigma$-algebra $\mathcal{B}_1^k$ is the most familiar example, but in general a Keisler graded probability space may have additional measurable sets which do not belong to the product $\sigma$-algebra.  These additional sets precisely correspond to the quasirandom graphs and hypergraphs \cite{MR3583029}.

More generally, we take $\mathcal{B}_I$ to be a measure space on $\Omega^I$, obtained from $\mathcal{B}_{|I|}$ in the natural way by choosing any bijection between $I$ and $\{1,2,\ldots,|I|\}$, and we have a corresponding measure $\mu_I$ on $\mathcal{B}_I$.  Since $\mathcal{B}_{|I|}$ and $\mu_{|I|}$ are symmetric, $\mathcal{B}_I$ and $\mu_I$ do not depend on the choice of bijection.

The $\sigma$-algebra $\mathcal{B}_k$ of all measurable sets has canonical sub-$\sigma$-algebras generated by cylinder intersection sets which are, in general, proper; these correspond exactly to the ``non-quasirandom'' sets (for various notions of quasirandomness).
\begin{definition}
  When $r<k$, $\mathcal{B}_{k,r}$ is the sub-$\sigma$-algebra of $\mathcal{B}_k$ generated by all $([k],{[k]\choose r})$-cylinder intersection sets where all components are elements of $\mathcal{B}_r$.

  More generally, when $\mathcal{D}$ is a sub-algebra of $\mathcal{B}_r$, we write $\mathcal{K}_{k,r}(\mathcal{D})$ for the sub-$\sigma$-algebra of $\mathcal{B}_k$ generated by all $(k,{[k]\choose r})$-cylinder intersection sets where all components are elements of $\mathcal{D}$.
  
  We say $\{(\Omega^k,\mathcal{B}_k,\mu_k)\}_{k\in\mathbb{N}}$ is \emph{countably approximated} if each, for each $k$, there is a countable algebra of sets $\mathcal{B}_k^0\subseteq\mathcal{B}^k$ such that:
\begin{itemize}
\item $\mathcal{K}_{k,r}(\mathcal{B}_r^0)\subseteq\mathcal{B}_k^0$ for all $r<k$,
\item the algebras $\mathcal{B}_k^0$ are symmetric,
\item whenever $B\in\mathcal{B}_k^0$, $r<k$, and $q\in\mathbb{Q}\cap(0,1)$, there is a $D\in\mathcal{B}_r^0$ with
  \[\{\vec x\in\Omega^r\mid \mu(B_{\vec x})<q\}\subseteq D\subseteq\{\vec x\in\Omega^r\mid\mu(B_{\vec x})\leq q\}.\]
\end{itemize}
\end{definition}
Ultraproducts of graphs are countably approximated, using the definable sets (in a large enough language) as the approximating sets. (It turns out that we \emph{cannot} quite expect to exactly close the algebras under level sets; see \cite{goldbring:_approx_logic_measure} for more on the approach here.)

We can think of $\mathcal{B}_{k,r}$ as being the sets of $k$-tuples which are ``explained by'' properties of $r$-tuples.

Since $\mathcal{B}_{k,r}$ is symmetric, we can also define $\mathcal{B}_{W,r}$ in the natural way---equivalently, as the image of $\mathcal{B}_{|W|,r}$ under any bijection of $W$ with $\{1,\ldots,|W|\}$, or as the $\sigma$-algebra generated by $(W,{W\choose r})$-cylinder intersection sets where the component $A_e$ belongs to $\mathcal{B}_e$.

\begin{definition}
  We define
  $t_S(\{A_e\}_{e\in S})=\mu(T_S(\{A_e\}_{e\in S}))$.  More generally, we define
  \[t_S(\{f_e\}_{e\in S})=\int \prod_{e\in S}f_e(\vec x_e)\,d\mu.\]

  We similarly define $t_H(G)=\mu(T_H(E))$, $t^{ind}_H(G)=\mu(T^{ind}_H(G))$, and $t_{W,c}(\Omega,\rho)=\mu(T_{W,c}(\Omega,\rho))$.
\end{definition}
There will be no confusion between these related definitions, since $t_S(\{A_e\})=t_S(\{\chi_{A_e}\})$.

\section{Removal and Induced Removal for Graphs}

In this section we prove graph removal, using this as a vehicle to introduce our notation and approach and prove some lemmas we will need for the more general results in later sections.

\subsection{Neighborhoods and Points of Density}

We would like to work with points of density of measurable functions---that is, points which behave like limits of the nearby points.  One difficulty is that a general probability measure space may not have a natural basis like the open balls.  We will fix this by brute force: we simply pick, more or less arbitrarily, a family of neighborhoods around each point which suffice for our purposes.\footnote{An alternative method, which plays a central role in the ``graphon'' approach to limit graphs \cite{MR3012035}, is to use the fact that every probability measure space is, in a suitable way, equivalent to the Lebesgue measure on the unit interval, and then use the usual notion of a point of density.  This is used, for instance, in \cite{MR2964622} to prove hypergraph regularity.}

More precisely, we will have, for each point $x\in \Omega$, a sequence $\mathcal{N}^j(x)$ of neighborhoods such that $\mu(\mathcal{N}^j(x))\rightarrow 0$.  (The analogous arrangement in the Lebesgue measure would take $\mathcal{N}^j(x)=B_{1/j}(x)$.)

Since we will need it later and the definition is the same, we will define a system of neighborhoods around tuples in $\Omega^r$ as well.

\begin{definition}
  When $\mathcal{D}^0$ is a countable collection of subsets of $\Omega^r$, a \emph{system of neighborhoods in $\mathcal{D}^0$} is a sequence of partitions, $\mathcal{N}=\{\mathcal{N}^j\}$ such that:
  \begin{itemize}
  \item each $\mathcal{N}^j$ is a finite partition of $\Omega^r$,
  \item when $i<j$, $\mathcal{N}^j$ refines $\mathcal{N}^i$,
  \item for every set $A\in\mathcal{D}^0$, there is a $j$ so that $A$ differs by measure $0$ from a union of elements of $\mathcal{N}^j$,
  \item $\lim_{j\rightarrow\infty}\max_{P\in\mathcal{N}^j}\mu(P)=0$.
  \end{itemize}
  We call $r$ the \emph{arity} of $\mathcal{N}$.

  We write $\mathcal{N}^\sigma$ for the $\sigma$-algebra generated by all sets in $\bigcup_j\mathcal{N}^j$.
\end{definition}

Since we will be working with partitions frequently, we introduce some notation. 

\begin{definition}
  When $\mathcal{N}^j$ is a partition of $\Omega^r$ and $\vec x\in\Omega^r$ is a point, we write $\mathcal{N}^j(\vec x)$ for the unique set $P\in\mathcal{N}^j$ such that $\vec x\in P$.  
\end{definition}


We should think of $\mathcal{N}$ as being a schema giving, for each tuple $\vec x_W$ and each number $j$, a set $\mathcal{N}^j(\vec x_W)$ which is the ``ball around the tuple $\vec x_W$''.

We want to lift partitions of $\Omega^r$ to partitions of $\Omega^k$ with $r<k$ in the obvious way---a partition of $\Omega^k$ is a cylinder intersection set coming from our partition of $\Omega^r$.

\begin{definition}
  When $\mathcal{N}^j$ is a partition of $\Omega^r$, $r\leq |W|$ and $\vec x_W\in\Omega_W$, we write $\mathcal{N}^j(\vec x_W)$ for the ${W\choose r}$-cylinder intersection set $T_{{W\choose r}}(\{\mathcal{N}^j(\vec x_e)\}_{e\in{W\choose r}})$.
\end{definition}

For instance, when $\mathcal{N}$ has arity $1$, it induces partitions of $\Omega^2$ into sets of the form $P\times Q$ where $P,Q\in\mathcal{N}^j$.

A system of neighborhoods $\mathcal{N}=\{\mathcal{N}^j\}$ give us a natural way to define density.
\begin{definition}
  Let $\mathcal{N}$ be a system of neighborhoods of arity $r$.  Given $k\geq r$ and $f:\Omega^k\rightarrow[0,1]$, define
  \[f^j_{\mathcal{N}}(\vec x)=\frac{1}{\mu(\mathcal{N}^j(\vec x))}\int_{\mathcal{N}^j(\vec x)}f(\vec x)\,d\mu\]
  whenever $\mu(\mathcal{N}^j(\vec x))>0$ and
  \[f^+_{\mathcal{N}}(\vec x)=\lim_{j\rightarrow\infty}f^j_{\mathcal{N}}(\vec x)\]
  wherever each $f^j_{\mathcal{N}}$ is defined and this limit exists.  We call $\vec x$ a \emph{point of density for $f$ in $\mathcal{N}$} if $f^+_{\mathcal{N}}(\vec x)$ exists and
  \[\lim_{j\rightarrow\infty}\frac{1}{\mu(\mathcal{N}^j(\vec x))}\int_{\mathcal{N}^j(\vec x)}|f^+_{\mathcal{N}}(\vec y)-f^+_{\mathcal{N}}(\vec x)|\,d\mu(\vec y)=0.\]

  When $A\subseteq\Omega^k$, we call $\vec x$ a point of density for $A$ in $\mathcal{N}$ if $\vec x$ is a point of density for $\chi_A$.

\end{definition}

One might have expected the definition of a point of density to be simply that
\[\lim_{j\rightarrow\infty}\frac{1}{\mu(\mathcal{N}^j(\vec x))}\int_{\mathcal{N}^j(\vec x)}|f(\vec y)-f(\vec x)|\,d\mu(\vec y)=0.\]
But take $E$ to be a quasi-random graph and let $f=\chi_E$ and $\mathcal{N}$ a system of neighborhoods of arity $1$; in this case, every positive measure neighborhood $\mathcal{N}^j(x_1,x_2)=\mathcal{N}^j(x_1)\times\mathcal{N}^j(x_2)$ has the property that half its points belong to $E$ and half do not belong to $E$, so there would be no points of density as all.

This is the fundamental difference from classical Lebesgue measure: because we are working in a Keisler graded probability space with quasi-random elements, we cannot expect most points in the graph to be near other points in the graph. However we will see that we can expect most points to have a well-defined density, and to be near other points with a similar density.

We will usually want, not just any point of density of $f$, but one where the density $f^+_{\mathcal{N}}$ is positive.

\begin{definition}
  We say $x$ is a \emph{positive point of density} of $f$ if $x$ is a point of density of $f$ with $f^+_{\mathcal{N}}(x)>0$.  When $E$ is a set, a \emph{positive point of density} of $E$ is a positive point of density of $\chi_E$.
\end{definition}

When $r=1$---that is, when $\mathcal{N}$ consists of sets of points---there are no particular symmetry issues.  In particular, when $f$ is a symmetric function (for instance, the characteristic function of a graph of hypergraph), every permutation of a point of density is also a point of density.  When $r>1$, we have to worry about whether the neighborhoods themselves are symmetric.
\begin{lemma}
  If $f$ is symmetric, each $\mathcal{N}^j$ is symmetric (that is, each permutation of a set in $\mathcal{N}^j$ is also in $\mathcal{N}^j$), and $\vec x$ is a point of density for $f$ in $\mathcal{N}$ then each permutation of $\vec x$ is also a point of density.
\end{lemma}

In general, $f^+_{\mathcal{N}}$ is $\mathbb{E}(f\mid\mathcal{K}_{k,r}(\mathcal{N}^\sigma))$.  More precisely, $\mathbb{E}(f\mid\mathcal{K}_{k,r}(\mathcal{N}^\sigma))$ is only defined up to the $L^2$ norm, so $f^+_{\mathcal{N}}$ is a natural representative of $\mathbb{E}(f\mid\mathcal{K}_{k,r}(\mathcal{N}^\sigma))$.


\begin{lemma}\label{thm:almost_every_lebesgue}
  
  For any measurable function $f:\Omega^k\rightarrow[0,1]$, $f^+_{\mathcal{N}}$ is defined almost everywhere, $f^+_{\mathcal{N}}=\mathbb{E}(f\mid\mathcal{K}_{k,r}(\mathcal{N})^\sigma)$, almost every $x$ is a point of density of $x$, and almost every point $x$ with $f(x)>0$ is a positive point of density of $f$.
\end{lemma}
\begin{proof}

  We first show that the functions $f^j_{\mathcal{N}}$ converge in the $L_2$ norm to $\mathbb{E}(f\mid\mathcal{K}_{k,r}(\mathcal{N}^\sigma))$.  For any $\epsilon>0$, we may choose $j_0$ large enough that $||\mathbb{E}(f\mid\mathcal{K}_{k,r}(\mathcal{N}^{j_0}))-\mathbb{E}(f\mid\mathcal{K}_{k,r}(\mathcal{N}^\sigma))||_{L^2}<\epsilon$.  Then whenever $j\geq j_0$, $\mathcal{N}^j$ refines $\mathcal{Q}$ up to measure $0$, so also $||f^j_{\mathcal{N}}-\mathbb{E}(f\mid\mathcal{K}_{k,r}(\mathcal{N}^\sigma))||_{L^2}<\epsilon$.

  To see that the pointwise limit is defined almost everywhere and that almost every point is a point of density, consider any $\epsilon>0$ and $\alpha<\beta$.  Let $g=\mathbb{E}(f\mid\mathcal{K}_{k,r}(\mathcal{N}^\sigma))$.  Choose $j_0$ large enough so that there is a set $S\in\mathcal{K}_{k,r}(\mathcal{N}^{j_0})$ so that $\mu(S\bigtriangleup \{\vec x\mid g(\vec x)\leq\alpha\})<\frac{\beta-\alpha}{1-\alpha}\epsilon$.  

  Consider all rectangles $R$ from $\bigcup_j\mathcal{N}^j$ which are contained in $S$ and such that the average of $f$ on $R$ is $\geq\beta$.  Since $\int_R g\,d\mu\geq \beta\mu(R)$ and $g\leq 1$, we must have $\{\vec x\in R\mid g(\vec x)>\alpha\}\geq\frac{\beta-\alpha}{1-\alpha}\mu(R)$, and therefore $\mu(R)<\epsilon$.    Therefore, once $j\geq j_0$, except for a set of measure $\epsilon$, if $g(\vec x)\leq\alpha$ then for all $j\geq j_0$, $f^j_{\mathcal{N}}(\vec x)\leq\alpha$ as well.  So the set of points with $g(\vec x)\leq\alpha$ but $\limsup f^j_{\mathcal{N}}(\vec x)>\alpha$ has measure $<\epsilon$.  Dually, we can show that the set of points with $g(\vec x)\geq\beta$ but $\liminf f^j_{\mathcal{N}}(\vec x)<\beta$ has measure $<\epsilon$.  Since this holds for all $\alpha,\beta$ and all $\epsilon$, for almost all $\vec x$ we have $f^+_{\mathcal{N}}(\vec x)=\lim f^j_{\mathcal{N}}(\vec x)=g(\vec x)$.

By the same argument, for any $\alpha$ and any $\delta>0$ we see that when $f^+_{\mathcal{N}}(\vec x)\leq\alpha$, except for a set of measure $<\epsilon$, for all sufficiently large $j$ we have $f^j_{\mathcal{N}}(\vec x)\leq\alpha+\delta$, and therefore since $\frac{1}{\mu(\mathcal{N}^j(\vec x))}\int_{\mathcal{N}^j(\vec x)}f(\vec x)\,d\mu=\frac{1}{\mu(\mathcal{N}^j(\vec x))}\int_{\mathcal{N}^j(\vec x)}f^+_{\mathcal{N}}(\vec x)\,d\mu$, the set of $\vec y\in\mathcal{N}^j(\vec x)$ with $f^+_{\mathcal{N}}(\vec y)\geq\alpha+\delta$ is small.  So almost every $\vec x$ is a point of density.

Finally, to see that almost every point with $f(x)>0$ has $f^+_{\mathcal{N}}(x)>0$, let $Z$ be the set of points where $f^+_{\mathcal{N}}(\vec x)=0$.  Since $f^+_{\mathcal{N}}$ is $\mathcal{K}_{k,r}(\mathcal{N})$-measurable, $Z$ belongs to the completion of $\mathcal{K}_{k,r}(\mathcal{N})$, so $0=\int_Zf^+\,d\mu=\int_Z f\,d\mu$, so the set of $\vec x\in Z$ where $f(\vec x)>0$ has measure $0$.
\end{proof}

\subsection{Counting and Graph Removal}

The next fact we need is that the quantity $t_S(\{f_e\}_{e\in S})$ depends only on the ``non-random'' part of the $f_e$.  In its simplest form, this says that if $E$ is a graph, $t_S(E)=t_S(\mathbb{E}(E\mid\mathcal{B}_{2,1}))$---that is, we can replace the graph $E$ with the function $\mathbb{E}(E\mid\mathcal{B}_{2,1})$ measuring the density of $E$ when counting graph densities.\footnote{In the graphon approach, this fact plays a central role: the object $\mathbb{E}(E\mid\mathcal{B}_{2,1})$ \emph{is} the graphon, and the basic theorems establish that for things like counting graph densities, this is all that is needed.}

We will state this fact in a very general way which will continue to serve us as we deal with $k$-graphs.

\begin{lemma}\label{thm:reduction}
  Let $\{f_e\}_{e\in S}$ be given and, for each $e$, let $\mathcal{D}_e$ be a $\sigma$-algebra of sets of $r$-tuples such that, for every $e_0\in S$, $|e_0|\geq r$ and either:
  \begin{itemize}
  \item $f_{e_0}$ is $\mathcal{D}_{e_0}$-measurable, or
  \item for every $e\in S\setminus\{e_0\}$ and each fixed $\vec x_{e\setminus e_0}$, the function $\vec x_{e_0}\mapsto f_e(\vec x_e)$ is $\mathcal{D}_{e_0}$-measurable.
  \end{itemize}
For each $e$, let $f'_e=\mathbb{E}(f_e\mid \mathcal{D}_e)$.  Then $t_S(\{f_e\})=t_S(\{f'_e\})$.
\end{lemma}
The general form allows the case where $S$ contains tuples of different sizes, and replaces $\mathcal{B}_{2,1}$ with a more general $\sigma$-algebra which may depend on the coordinate $e$; most commonly, we will have $\mathcal{D}_e=\mathcal{K}_{e,r}(\mathcal{D})$ for a fixed $\sigma$-algebra $\mathcal{D}$.

We need some requirement that $\mathcal{D}_e$ is large enough.  For example, when we turn to $3$-graphs, we might initially try $\mathcal{D}_e=\mathcal{B}_{3,1}$ while $S\subseteq {W\choose 3}$.  Working only with $\mathcal{B}_{3,1}$ amounts to working with weak hypergraph regularity \cite{MR1099803}, which is known to suffice when $S$ is \emph{linear}---that is, when $|e\cap e'|\leq 1$ for any distinct $e,e'\in S$ \cite{MR2595699,MR2864650}.  But when the elements of $S$ can overlap more generally, we need to work with a larger $\sigma$-algebra, for instance $\mathcal{B}_{3,2}$.  This is precisely what the second case of the lemma requires: that the ``overlaps'' with the other functions is already measurable with respect to $\mathcal{D}_{e_0}$.
\begin{proof}
We show by induction on $|T|$, where $T\subseteq S$,
that
\[t_S(\{A_e\})=\int\prod_{e\in T}f'_e(\vec x_e)\prod_{e\in S\setminus T}f_e(\vec x_e)\,d\mu.\]
When $T=S$, this gives the desired claim.

When $T=\emptyset$, the statement is trivial.

Suppose the inductive hypothesis holds for $T$ and that $e_0\in S\setminus T$.  Then we have
\begin{align*}
  t_S(\{f_e\})
  &=\int f_{e_0}(\vec x_{e_0})\prod_{e\in T}f'_{e}(\vec x_{e})\prod_{e\in S\setminus T\cup\{e_0\}}f_{e}(\vec x_{e})\,d\mu.
\end{align*}
For a fixed $\vec x_{W\setminus e_0}$, consider the function
\[h(\vec x_{e_0})=\prod_{e\in T}f'_{e}(\vec x_{e})\prod_{e\in S\setminus T\cup\{e_0\}}f_{e}(\vec x_{e}).\]
Each term in the product is $\mathcal{D}_{e_0}$-measurable, so $h$ is $\mathcal{D}_{e_0}$-measurable as well.  Therefore
\begin{align*}
  t_S(\{f_e\})
  &=\int\mathbb{E}(f_{e_0}\mid \mathcal{K}_{e_0,r}(\mathcal{D}))(\vec x_{e_0})\prod_{e\in T}f'_{e}(\vec x_{e})\prod_{e\in S\setminus T\cup\{e_0\}}f_{e}(\vec x_{e})\,d\mu\\
  &=\int f'_e(\vec x_e)\prod_{e'\in T}f'_{e'}(\vec x_{e'})\prod_{e'\in S\setminus T\cup\{e\}}f'_{e'}(\vec x_{e'})\,d\mu
\end{align*}
which gives the inductive claim.
\end{proof}

The next result should be seen as our version of the graph counting lemma.  Typically, a graph counting lemma would say something like the following:
\begin{quote}
  Suppose $S\subseteq{W\choose 2}$ and that for each $w\in W$, we have a set $P_w\subseteq\Omega$ such that, for each pair $(w,w')\in S$, $\frac{1}{\mu(P_w\times P_{w'})}\int_{P_w\times P_{w'}}f_{\{w,w'\}}\,d\mu>\epsilon$, and also $f_{\{w,w'\}}$ is suitably quasi-random between $P_w$ and $P_{w'}$.  Then $t_S(\{f_e\})>0$.
\end{quote}
In our setting, we are able to ``take the limit'' as the size of the sets $P_w$ approaches $0$: instead of sets $P_w$, we will be able to work with individual points $x_w$ (and, therfore, sufficiently small neighborhoods $\mathcal{N}^j(x_w)$).  The requirement that $f_{\{w,w'\}}$ be suitably quasi-random becomes the requirement that $(x_w,x_{w'})$ be a points of density, and the requirement that $f_{\{w,w'\}}$ have positive density becomes the requirement that $f^+_{\{w,w'\}}(x_w,x_{w'})>0$.

This result is the first place we restrict ourselves to the graph case---that is, to requiring that $\mathcal{N}$ be a system of neighborhoods with arity $1$---since the hypergraph version requires more work.

\begin{theorem}\label{thm:graph_counting}
  Let $W$ be a finite set, let $\mathcal{N}$ be a system of neighborhoods of arity $1$, and let $S$ be a collection of subsets of $W$. 
  Suppose that, for each $e\in S$, either:
  \begin{itemize}
  \item $f_e$ is $\mathcal{K}_{e,1}(\mathcal{N})$-measurable, or
  \item for every $e'\in S\setminus\{e\}$, the function $\vec x_e\mapsto \int f_{e'}(\vec x_{e\cup e'})\,d\mu(\vec x_{e'\setminus e})$ is $\mathcal{K}_{e,1}(\mathcal{N})$-measurable.
  \end{itemize}

  Suppose that $\vec x_W\in T_S(\{f_e\})$ is such that, for each $e\in S$, $\vec x_e$ is a positive point of density of $f_e$.  Then $t_S(\{f_e\})>0$.
\end{theorem}
The basic idea of the proof is that we may ``blow up'' each individual point $x_w$ into a small ball $\mathcal{N}^j(x_w)$, and then use the fact that each $\vec x_e$ is a point of density $f_e$ to find many copies of $W$ between these small balls.
\begin{proof}
Choose some $\epsilon\leq\min_{e\in S}f^+_e(\vec x_e)$.

Since each $\vec x_e$ is a point of density, we may choose some $j$ large enough that, for each $e\in S$,
\[\frac{1}{\mu(\mathcal{N}^j(\vec x_e))}\mu(\{\vec y_e\in\mathcal{N}^j(\vec x_e)\mid f^+_e(\vec y_e)\geq\epsilon/2\})\geq 1-\frac{1}{|S|}.\]
Therefore also
\[\frac{1}{\mu(\mathcal{N}^j(\vec x_W))}\mu(\{\vec y_W\in\mathcal{N}^j(\vec x_W)\mid f^+_e(\vec y_e)\geq\epsilon/2\})\geq 1-\frac{1}{|S|}.\]
Note that this depends on the fact that the arity of $\mathcal{N}$ is $1$, because this ensures that $\mathcal{N}^j(\vec x_W)=\mathcal{N}^j(\vec x_e)\times\mathcal{N}^j(\vec x_{W\setminus e})$.

Therefore
\[\gamma=\mu(\{\vec y_W\in\mathcal{N}^j(\vec x_W)\mid \text{ there is some }e\in S\text{ such that }f^+_e(\vec y_e)<\epsilon/2\})<|S|\frac{1}{|S|}=1,\]
so, using Lemma \ref{thm:reduction} (with $\mathcal{D}_e=\mathcal{B}_{2,1}$ for all $e$),
\[t_S(\{f_e\})=t_S(\{f^+_e\})\geq\frac{\epsilon^{|S|}}{2^{|S|}}(1-\gamma)>0.\]
\end{proof}

\begin{theorem}[Graph Removal]\label{thm:graph_removal}
  Suppose $H=(W,F)$ is a finite graph and $G=(\Omega,E)$ is a graph with a countably approximated atomless Keisler graded probability space on $\Omega$ with $E\in\mathcal{B}^0_2$.  If $t_H(G)=0$ then there is a symmetric $E'\subseteq E$ such that $E\setminus E'$ is a measure $0$ set contained in an intersection of sets in $\mathcal{B}^0_2$ and, taking $G'=(\Omega,E')$, $T_H(G')=\emptyset$.
\end{theorem}
\begin{proof}
  Choose $\mathcal{N}_1$ so that every set in $\mathcal{B}_1^0$ is a finite union of sets in $\bigcup_j\mathcal{N}_1^j$.  Let $E'\subseteq E$ consist of the positive points of density of $E$.  By Lemma \ref{thm:almost_every_lebesgue}, $\mu(E\setminus E')=0$.  If $T_H(E')\neq\emptyset$ then any $\vec x_W\in T_H(E')$ satisfies the conditions of the previous lemma, and so $t_H(E)>0$.

  Since $E\in\mathcal{B}^0_2$ and $E\setminus E'$ is contained in an intersection of finite unions of rectangles from $\mathcal{B}_1^0$, $E\setminus E'$ is contained in an intersection of sets in $\mathcal{B}^0_2$.
\end{proof}

\begin{cor}
  For every finite graph $H=(W,F)$ and every $\epsilon>0$ there is a $\delta>0$ so that whenever $G=(V,E)$ is a graph with $t_H(G)<\delta$, there is a symmetric $E'\subseteq E$ with $|E\setminus E'|<\epsilon|V|^2$ such that, taking $G'=(V,E')$, $T_H(G')=\emptyset$.
\end{cor}
\begin{proof}[Sketch]
  The proof is standard (see \cite{goldbring:_approx_logic_measure}), but we include the outline here.  Suppose the statement were false, so let $H=(W,F)$ and $\epsilon>0$ be a counterexample.  Then for each $n$, there is a $G_n=(V_n,E_n)$ with $t_H(G_n)<1/n$, but so that no symmetric $E'\subseteq E_n$ with $|E\setminus E'|<\epsilon|V_n|^2$ is $H$-free.  Note that $|V_n|\rightarrow\infty$ (otherwise $t_H(G_n)<1/n$ implies $T_H(G_n)=\emptyset$ for $n$ large enough).

  Let $(\Omega,E)$ be an ultraproduct of the sequence $G_n$.  Take the Keisler graded probability space generated by the definable sets, with the Loeb measure.  Let $E'$ be given by the previous lemma.  Then $E\setminus E'$ is contained in an intersection of definable sets, so choosing some definable set $Z_m$ large enough, $E\setminus Z_m$ is $H$-free and $Z_m$ has measure $<\epsilon$.  By the \L{o}\'s Theorem, for infinitely many $n$, we have $(V_n,E_n\setminus Z_m)$ is also $H$-free and $Z_m$ has measure $<\epsilon$.  (Where, by $Z_m$, we mean the interpretation of the definable set $Z_m$ in the structure $G_n$.)  But this is contradicts the choice of the $G_n$.
\end{proof}

\subsection{Induced Graph Removal}

When we prove induced graph removal, we have a new issue to deal with: when $\vec x$ is not a point of density, we cannot simply exclude the point from $E$, because, by doing so, we might end up creating an induced copy $\vec x$ where one of the non-edges of $\vec x$ is an element we removed from $E$.

Instead, we adopt a more complicated strategy.  We choose $j$ large, so that $\mathcal{N}^j$ gives a partition of $\Omega$ into very small pieces.  We will then choose, from each element $P$ of $\mathcal{N}^j$, a representative $a_P\in P$, uniformly at random.  Since we are only choosing finitely many such elements, with probability $1$, all the pairs $(a_P,a_{P'})$ with $P\neq P$ are points of density.  We then modify $E$ to match $(a_P,a_{P'})$ on $P\times P'$; that is, we define a new graph $E'$: if $(a_P,a_{P'})\in E$ then we place all of $P\times P'$ in $E'$, while if $(a_P,a_{P'})\not\in E$ then we exclude all of $P\times P'$ from $E'$.  If we choose $j$ large enough, we will be able to show that, for most choices of the representatives $a_P$, $\mu(E\bigtriangleup E')$ is small.

This leaves us with a new problem: what to do with the ``diagonal components'' $P\times P$.  When $j$ is large, these diagonals have small measure, so we can put them in or out of $E'$ as convenient.  On the other hand, we cannot guarantee that $(a_P,a_P)$ is a point of density.

\begin{theorem}[Induced Graph Removal]\label{thm:induced_removal}
  Suppose $H=(W,F)$ is a finite graph and $G=(\Omega,E)$ is a graph with a countably approximated atomless Keisler graded probability space on $\Omega$ with $E\in\mathcal{B}^0_2$. For each $\epsilon>0$ there is a symmetric $E'\in\mathcal{B}^0_2$ such that $\mu(E'\bigtriangleup E)<\epsilon$ and for any $H$ with $t_H^{ind}(E)=0$, $T_H^{ind}(E')=\emptyset$.
\end{theorem}
\begin{proof}
  Let $f=\chi_E$.  Choose $j$ large enough that the set of pairs $(x_1,x_2)$ for which $\chi^j_E$ has not converged to within $\epsilon/3$ of its limit has measure at most $\epsilon/3$, and so that $\sum_{P\in\mathcal{N}^j}\mu(P\times P)<\epsilon/3$.

  We consider a partition of ${\Omega\choose 2}$ into three sets: $E_1=\{(x,y)\mid f^+(x,y)=1\}$, $E_0=\{(x,y)\mid f^+(x,y)=0\}$, and $E_{1/2}=\{(x,y)\mid 0<f^+(x,y)<1\}$.  (There is also a set of measure $0$ where $f^+(x,y)$ is undefined.)  We may think of these as the interior of $E$, the interior of the complement of $E$, and a boundary of points near both $E$ and the complement of $E$.  

  Suppose that, for each $P\in\mathcal{N}^j$ with $\mu(P)>0$, we choose a point $a_P\in P$ uniformly at random.  Then, with positive probability:
  \begin{itemize}
  \item the set of points contained in $P\times P'$ with $P\neq P'$ and such that $|f^+(a_P,a_{P'})-f^j(a_P,a_{P'})|\geq\epsilon/3$ has measure at most $\epsilon/3$,
  \item each $(a_P,a_{P'})$ is a point of density for each of $E_1,E_0,E_{1/2}$ and is a positive point of density for the set it belongs to.
  \end{itemize}

  Next we prepare to deal with elements of the sets $P\times P$. What we want to do is choose many points near each $a_P$; when we choose one point $b_{P,i}$ near $a_P$ and one point $b_{P',j}$ near $a_{P'}$ with $P\neq P'$, we can ensure, with high probability, that $(b_{P,i},b_{P',j})$ is similar to $(a_P,a'_P)$. When we take two points near the same $a_P$, $b_{P,i}$ and $b_{P,j}$ with $i\neq j$, we have no control over what happens. However, by applying Ramsey's Theorem (many times), we can at least ensure that the behavior does not depend on the particular choice of $i$ and $j$.

Formally, we will choose these points by applying our counting lemma to a suitable graph.  We may let $A=\{a_P\mid P\in\mathcal{N}^j, \mu(P)>0\}$.  For any $d$, let us consider the \emph{colored $d$-blowup} of $A$, which we define to be the $\{0,1/2,1\}$-colored graph $(A_d,c_d)$ where:
  \begin{itemize}
  \item $A_d=A\times[d]$,
  \item $\dom(c_d)=\{((a,i),(a',j))\mid a\neq a'\}$,
  \item when $a\neq a'$ and $(a,a')\in E_z$, $c_d((a,i),(a',j))=z$.
  \end{itemize}
Observe that Theorem \ref{thm:graph_counting} applies to $(A_d,c_d)$, so $t_{(A_d,c_d)}(\{E_z\}_{z\in\{0,1/2,1\}})>0$.

  When $v:A\rightarrow\{0,1,1/2\}$, the \emph{$v$-homogeneous completion} of $(A_d,c_d)$ is the colored graph $(A_d,c^v_d)$ where $c_d\subseteq c^v_d$ and, for $i\neq j$, $c^v_d((a,i),(a,j))=v(a)$.

  Take $m$ sufficiently large and consider any copy $\vec b_{A_m}$ of the colored $m$-blowup of $A$ in $(\Omega,E)$.  (This means that for each pair $((a,i),(a',j))\in{A_d\choose 2}$ with $a\neq a'$, $(b_{(a,i)},b_{(a,j)})\in E_z$ if and only if $(a,a')\in E_z$, and we make no commitments about which of the three sets $((a,i),(a',i))$ belongs to.)  Applying Ramsey's Theorem once for each $a\in A$, there is a $v$ and a sub-copy $\vec b_{A_d}$ of $\vec b_{A_m}$ which is a copy of $(A_d,c^v_d)$.

  Since there are only finitely many $v$, this means that for each $d$ there some $v$ so that $t_{(A_d,c^v_d)}(E)>0$.  Furthermore, if $d<d'$, we have $t_{(A_{d'},c^v_{d'})}(E)\leq t_{(A_d,c^v_d)}(E)$.  Therefore there must be some $v$ so that, for all $d$, $t_{(A_d,c^v_d)}(E)>0$.
  
  Finally, we have to deal with the case where $P\in\mathcal{N}^j$ has measure $0$.  To deal with this, we assign to every element $P\in \mathcal{N}^j$ a corresponding element $Q_P\in\mathcal{N}^j$, and we will always treat elements of $P$ as if they were really in $Q_P$.  For any $P\in\mathcal{N}^j$ with $\mu(P)=0$, choose some $Q_P\in\mathcal{N}^j$ with $\mu(Q_P)>0$.  When $\mu(P)>0$, take $Q_P=P$.  So for almost every point, $Q_P=P$, but there are a measure $0$ set of exceptional points\footnote{We could have tweaked our definition of a partition to avoid this case, but when we go on to hypergraphs, this case will be unavoidable, and the exceptional points will have small but positive measure.} where $Q_P\neq P$.

   We define $E'$ as follows:
  \begin{itemize}
  \item for $P\neq P'$, if $(a_{Q_P},a_{Q_{P'}})\in E_0$, let $E'\cap(P\times P')=\emptyset$,
  \item for $P\neq P'$, if $(a_{Q_P},a_{Q_{P'}})\in E_1$, let $P\times P'\subseteq E'$,
  \item for $P\neq P'$, if $(a_{Q_P},a_{Q_{P'}})\in E_{1/2}$, let $E'\cap (P\times P')=E\cap(P\times P')$,
  \item if $v(a_{Q_P})=1$ then $(P\times P)\subseteq E'$,
  \item if $v(a_{Q_P})=0$ then $E'\cap(P\times P)=\emptyset$,
  \item if $v(a_{Q_P})=1/2$ then $E'\cap(P\times P)=E\cap(P\times P)$.
  \end{itemize}

  Consider any graph $H=(W,F)$ such that $T^{ind}_H(E')\neq\emptyset$.  Take any $\vec x_W\in T^{ind}_H(E')$.  For each $w\in W$, let $P_w\in\mathcal{N}^j=Q_{\mathcal{N}^j}(x_w)$.  Note that we may have $P_w=P_{w'}$ even when $w\neq w'$, so fix an ordering $W=\{w_1,\ldots,w_{|W|}\}$.

  We have $t_{(A_{|W|},c^v_{|W|})}(\{E_z\}_{z\in\{0,1,1/2\}})>0$, so we may choose a copy $\vec y_{A_{|W|}}$ where all pairs are points of positive density for $E$ if they belong to $E_1\cup E_{1/2}$ and for $\overline{E}$ if they belong to $E_0\cup E_{1/2}$.

  Take $\vec z_{w_i}=\vec y_{(a_{P_{w_i}},i)}$.  For each pair $w_i\neq w_j$, observe that $(z_{w_i},z_{w_j})$ is a positive point of density for $E$ if $(w_i,w_j)\in F$ and for $\overline{E}$ if $(w_i,w_j)\not\in F$---to see this, suppose $(w_i,w_j)\in F$ (the case where $(w_i,w_j)\not\in F$ is symmetric):
  \begin{itemize}
  \item if $P_{w_i}\neq P_{w_j}$ then, since $(x_{w_i},x_{w_j})\in E'$, we have $(a_{P_{w_i}},a_{P_{w_j}})\not\in E_0$, so $(y_{a_{P_{w_i}},i},y_{a_{P_{w_j},j}})\in E_1\cup E_{1/2}$ and is therefore a positive point of density for $E$,
  \item if $P_{w_i}=P_{w_j}$ then, since $(x_{w_i},x_{w_j})\in E'$, we have $v(a_{P_{w_i}})\neq 0$, so again $(y_{a_{P_{w_i}},i},y_{a_{P_{w_i},j}})\in E_1\cup E_{1/2}$, and is therefore a positive point of density for $E$.
  \end{itemize}

  Therefore we may apply Theorem \ref{thm:graph_counting} to $\vec z_{w_i}$ to show that $t_H(E)>0$.

  It remains to show that $\mu(E\bigtriangleup E')<\epsilon$.  Observe that of $(x_w,x_{w'})\in E\bigtriangleup E'$ then, letting $P=\mathcal{N}^j(x_w)$ and $P'=\mathcal{N}^j(x_{w'})$, one of the following holds:
  \begin{enumerate}
  \item $P=P'$,
  \item $\mu(P)=0$ or $\mu(P')=0$,
  \item $|f^+(a_{P},a_{P'})-f^j(a_P,a_{P'})|>\epsilon/3$,
  \item $f^j(a_P,a_{P'})\geq 1-\epsilon/3$ and $(x_w,x_{w'})\not\in E$, or
  \item $f^j(a_P,a_{P'})< \epsilon/3$ and $(x_w,x_{w'})\in E$.
  \end{enumerate}
  The first case accounts for measure at most $\epsilon/3$, the second case for measure $0$, the third case for measure at most $\epsilon/3$, and the last two can each account for at most an $\epsilon/3$ proportion of each component $P\times P'$, so at most $\epsilon/3$ in total.  So $\mu(E\bigtriangleup E')<\epsilon$.
\end{proof}

\begin{cor}
  For every finite graph $H=(W,F)$ and every $\epsilon>0$ there is a $\delta>0$ so that whenever $G=(V,E)$ is a graph with $t^{ind}_H(G)<\delta$, there is a symmetric $E'\subseteq E$ with $|E\setminus E'|<\epsilon|V|^2$ such that, taking $G'=(V,E')$, $T^{ind}_H(G')=\emptyset$.
\end{cor}


\section{Hypergraphs}

\subsection{Sequences of Neighborhoods}

In order to extend the arguments above to hypergraphs, we need to deal with an additional complication.  When $G=(\Omega,E)$ and $H=(W,F)$ are graphs and we consider the product $t_H(G)=\int\prod_{(w,w')\in F}\chi_E(x_w,x_{w'})\,d\mu$, the distict terms in the product only overlap on a single coordinate.  The crucial step is that in Lemma \ref{thm:reduction}, when we look at a single edge $e_0=(w_0,w'_0)\in F$, the ``overlaps'' with other edges in $F\setminus\{e_0\}$ share at most one coordinate, and are therefore $\mathcal{B}_{2,1}$-measurable.  This means that we are able to use Lemma \ref{thm:reduction} (in the proof of Theorem \ref{thm:graph_counting}) to replace $E$ with $\mathbb{E}(\chi_E\mid\mathcal{B}_{2,1})$.

When $G=(\Omega,E)$ and $H=(W,F)$ are $3$-graphs, however, the product $t_H(G)=\int\prod_{(w,w',w'')\in F}\chi_E(x_w,x_{w'},x_{x''})\,d\mu$ has terms which can share two coordinates.  If we try to carry out a proof analogous to the proof of Theorem \ref{thm:graph_counting}, we are only able to reduce $E$ to $\mathbb{E}(\chi_E\mid\mathcal{B}_{3,2})$.  $\mathbb{E}(\chi_E\mid\mathcal{B}_{3,2})$, however, is ``graph-like''---it is described in terms of two coordinates at a time, like a graph.

This leads us to an iterated process where, at each step, we reduce the number of coordinates by one.  This means we need to consider, not a single system of neighborhoods, but a sequence of then: we will have a sequence of systems of neighborhoods, $\mathcal{N}_d,\ldots,\mathcal{N}_1$, and we will consider not just the neighborhoods $\mathcal{N}_d^j(\vec x)$, but how these neighborhoods sit in the neighborhood $\mathcal{N}^j_d(\vec x)\cap\mathcal{N}_{d-1}^{j'}(\vec x)$ with $j'\gg j$, and so on.

In this section we will set up all the general machinery.  For concreteness, we'll focus on the case needed to prove induced hypergraph removal, which means we will focus on the case where $\vec x$ is a $k$-tuple and we consider systems of neighborhoods $\mathcal{N}_{k-1},\ldots,\mathcal{N}_1$ where $\mathcal{N}_i$ has arity $i$.  We will refer to this, throughout this section, as the \emph{standard example}.  In particular, note that this example illustrates that in the intersection $\mathcal{N}^j_d(\vec x)\cap\mathcal{N}_{d-1}^{j'}(\vec x)$, the set $\mathcal{N}^j_d(\vec x)$ is ``more complicated'' (for example, it is defined using sets of arity $d$) while the set $\mathcal{N}_{d-1}^{j'}(\vec x)$ is ``finer'' (since $j'\gg j$, we are working with a much finer partition of $\Omega^{d-1}$).  So we are looking at neighborhoods which use ``some high complexity information and a lot of low complexity information''.

We will nonetheless work, where possible, with general systems of neighborhoods, since this is the case we will use in the next section.  (In the next section, $\mathcal{N}_{i+1}$ will have arity $i$, and $\mathcal{N}_1$ will consist only of intervals.)

For this purpose, we identify the property we need a sequence of systems of neighborhoods to have to be workable.  (For instance, when $k>3$, we cannot use a sequence of neighborhoods of arity $k-1$ followed immediately by a sequence of arity $1$).  We need some property that guarantees that the $\mathcal{N}_{i+1}$ is ``not too much more complicated'' than $\mathcal{N}_i$, and it should be related to the ``computability of overlaps'' clause from Theorem \ref{thm:reduction}.  The general property we need is given by the following definition.
\begin{definition}
  If $\mathcal{D}$ is a $\sigma$-algebra of sets of $s$ tuples, $\mathcal{C}$ is a $\sigma$-algebra of sets of $r$-tuples, and $s\leq r$, we say $\mathcal{C}$ is \emph{properly aligned} in $\mathcal{D}$ if, for any $C\in\mathcal{C}$ and any $c$ with $1\leq c\leq r$, the function
  \[f(x_1,\ldots,x_r)=\int \chi_C(y_1,\ldots,y_c,x_{c+1},\ldots,x_r)\,d\mu\]
  is $\mathcal{K}_{r,s}(\mathcal{D})$-measurable.

  We say a sequence of $\sigma$-algebras $\mathcal{D}_d,\ldots,\mathcal{D}_1$ where $\mathcal{D}_i$ is a $\sigma$-algebra of sets of $r_i$-tuples, is \emph{properly aligned} if:
  \begin{itemize}
  \item $r_1=1$,
  \item $r_i\leq r_{i+1}$ for each $i<d$, and
  \item $\mathcal{D}_{i+1}$ is properly aligned in $\mathcal{D}_i$ for each $i<d$.
  \end{itemize}
\end{definition}

Of course, the standard example itself is properly aligned.
\begin{lemma}
  The sequence of $\sigma$-algebras $\mathcal{B}_d,\mathcal{B}_{d-1},\ldots,\mathcal{B}_1$ is properly aligned.
\end{lemma}
\begin{proof}
  Since $r_i=i$, the first two conditions are immediate.  If $C\in\mathcal{B}_{i+1}$ then the function $f(x_1,\ldots,x_{i+1})=\int \chi_C(y_1,\ldots,y_c,x_{c+1},x_{i+1})\,d\mu$ depends only on $(x_{c+1},\ldots,x_{i+1})$, and is therefore $\mathcal{K}_{i+1,i+1-c}(\mathcal{B}_{i-c})\subseteq\mathcal{B}_{i+1,i}(\mathcal{B}_i)$-measurable.
\end{proof}

When dealing with graphs, although we stated things in terms of tuples $\vec x_W$, we were really interested in the collection of infinitesimal neighborhoods $\{\lim_{j\rightarrow\infty}\mathcal{N}^j(x_w)\}_{w\in W}$.  In the graph setting, however, we could ignore the distinction between a point and its infinitesimal neighborhood.

For hypergraphs, though, we need to consider multiple layers of infinitesimal neighborhoods: in the standard example, a pair $(x_1,x_2)$ has a pair of infinitesimal neighborhoods $\lim_{j\rightarrow\infty}(\mathcal{N}^j_1(x_1), \mathcal{N}^j_1(x_2))$ and then an infinitesimal neighborhood of pairs $\lim_{j\rightarrow\infty}\mathcal{N}^j_2(x_1,x_2)$.  The problem is that specifying an actual tuple of points pins down all these infinitesimal neighborhoods simultaneously.  But there could be distinct pairs $(x_1,x_2),(y_1,y_2)$ with $(\mathcal{N}^j_1(x_1), \mathcal{N}^j_1(x_2))=(\mathcal{N}^j_1(y_1), \mathcal{N}^j_1(y_2))$ for all $j$, but $\mathcal{N}^j_2(x_1,x_2)\neq\mathcal{N}^j_2(y_1,y_2)$ for some $j$---that is, a pair of infinitesimal neighborhoods of points might (and, in general, does) partition into many neighborhoods of pairs.

Now, however, we need to separate these notions properly.  We borrow model-theoretic terminology, referring to infinitesimal neighborhoods as \emph{types}.
\begin{definition}
  When $\mathcal{N}$ is a system of neighborhoods with arity $r$, an $\mathcal{N}$-type is a decreasing sequence $P_1\supseteq P_2\supseteq\cdots$ with each $P_j\in\mathcal{N}^j$ a non-empty set.  When $p=\{P_j\}$ is a type, we write $p(j)=P_j$.  For any $\vec x\in\Omega^r$, we write $tp_{\mathcal{N}}(\vec x)=\{\mathcal{N}^j(\vec x)\}$.
\end{definition}

There are two different perspectives on types which it will be useful to keep in mind below.  The simpler perspective is that a type is, essentially, a $G_\delta$-set (more precisely, a distinguished presentation of a $G_\delta$-set): the type is giving us the set of points $\bigcap_j P_j$, and dealing with $\mathcal{N}_1(x)$ rather than $x$ is a way of ``zooming out'' from $x$ to all the points infinitesimally close to it.

In particular, if we fix two $\mathcal{N}_1$-types $\mathcal{N}_1(x_1),\mathcal{N}_1(x_2)$, we are fixing two sets, and so the product $\mathcal{N}_1(x_1)\times\mathcal{N}_1(x_2)$ is itself a rectangle.  Although this rectangle has measure $0$, we can hope that it behaves like a limit of the positive measure rectangles $\mathcal{N}_1^j(x_1)\times\mathcal{N}_1^j(x_2)$.  For instance, if $E$ is a random graph, we might expect that $\mathcal{N}_1^j(x_1)\times\mathcal{N}_1^j(x_2)$ contains both pairs belonging to $E$ and pairs not belonging to $E$.  Indeed, we will see that almost all points belong to types which do behave like the limit of the positive measure types that approximate them.

There is a technical subtlety: perhaps there is a failure of compactness and the intersection $\bigcap P_j$ happens to be empty, even though each finite intersection is non-empty.  In practice, we always care more about the approximations to the set than the actual intersection: it the intersection happened to be empty, we could always fill in a point inside it.  Indeed, ultraproducts are \emph{saturated} which, in particular, ensures that each type is non-empty.

This suggests the second perspective: we can think of the types themselves as being points, in a different but related space.  That is, instead of working with the space $\Omega$ of points, we can work with a space $\Omega_1$ where an element of $\Omega_1$ is a $\mathcal{N}_1$-type, and we have a measurable function $tp:\Omega\rightarrow\Omega_1$.  We will not explicitly use this second perspective, but it may be useful to keep in mind.\footnote{This second perspective also an explicit connection to the graphon-based approaches to regularity, as in \cite{Coregliano2019SemanticLO,MR2964622}.  These approaches avoid the use of a Keisler graded probability space by taking our spaces $\Omega^r$ with $r>1$ and decomposing the ``non-product'' content into a separable factor.  For instance, where we work with $\Omega^2$, they use a ternary product $\Omega^3$, where the first two components represent copies of $\Omega$ while the third contains the part of $\Omega^2$ which is not measurable with respect to $\mathcal{B}_{2,1}$.  Types give an alternate construction of this: we can see that the map $tp_1:\Omega^2\rightarrow\Omega_1^2$ given by $tp_1(x,y)=(\mathcal{N}_1(x),\mathcal{N}_1(y))$ is inadequate---for instance, if $E$ is a random graph on $\Omega$, there is no $E_*\subseteq\Omega_1^2$ with $E=tp_1^{-1}(E_*)$.  Instead, the correct map is $tp:\Omega^2\rightarrow\Omega_1^2\times\Omega_2$, where $\Omega_2$ the space of $\mathcal{N}_2$-types; $\Omega^2$ is a Keisler graded probability space, but $\Omega_1^2\times\Omega_2$ is an ordinary measure-theoretic product.}

When $\vec x_W$ is a tuple, we want to consider the $\mathcal{N}$-type of $\vec x_W$, by which we mean the $\mathcal{N}$-types of all size $r$ subsets of $W$.  Slightly more generally, if $S\subseteq{W\choose r}$, we need to consider the collection of $\mathcal{N}$-types precisely for those $e\in S$.  (The case we will need this for is that, if $x_w=x_{w'}$, we will want to ignore those $e\in{W\choose r}$ which contain both $x_w$ and $x_{w'}$.)
\begin{definition}
  When $r\leq |W|$ is the arity of $\mathcal{N}$ and $S\subseteq{W\choose r}$, an $\mathcal{N}$-$S$-type is a tuple $\vec p_S=\{\vec p_e\}_{e\in S}$ such that for each $e\in S$, $\vec p_e$ is an $\mathcal{N}$-type and, for each $j$, $\vec p_S(j)=T_{S}(\{\vec p_e(j)\})$ is non-empty.

  For any point $\vec x_W$, letting $S=\mathcal{R}(\vec x_W)$, there is a corresponding $\mathcal{N}$-$S$-type $tp(\vec x_W)$ given by $(tp_{\mathcal{N}}(\vec x_W))_e=tp_{\mathcal{N}}(\vec x_e)$.
\end{definition}
The only case we will need is where $S={W\choose r}\setminus\mathcal{R}(\vec x_W)$ (or an analog replacing $\vec x_W$ with $\mathcal{N}_1$-types).   Since tuples with repeated coordinates are an exceptional case with measure $0$, they will not be needed until we deal with induced hypergraph removal.

Note that $f^+_{\mathcal{N}}(\vec x)$ depends only on the type of $\vec x$, not on the particular point, and so being a point of density is a property of the type: if $\vec x$ is a point of density for $f$ in $\mathcal{N}$ then so is every $\vec x$ in $tp_{\mathcal{N}}(\vec x)$.

Finally, we need our most general definition: we have a sequence of systems of neighborhoods $\mathcal{N}_d,\ldots,\mathcal{N}_1$ and want to consider the $\mathcal{N}_i$-type of a $\vec x_W$-tuple for all $i$ simultaneously.
\begin{definition}
  When $\vec p_1=\{\vec p_{1,w}\}_{w\in W}$ is a $\mathcal{N}_1$-${W\choose 1}$-type, we write $\mathcal{R}_r(\vec p_1)$, the \emph{tuples of length $r$ with repeated elements} for the set of tuples $e\in{W\choose r}$ such that there are $w,w'\in e$ with $w\neq w'$ and $p_{1,w}=p_{1,w'}$.  When there are no repeated types, we will write $\mathcal{R}(\vec p_1)=\emptyset$ (omitting the subscript $r$).
  
  When $\mathcal{N}_d,\ldots,\mathcal{N}_1$ is a sequence of systems of neighborhoods for each $i$, a $\mathcal{N}_d,\ldots,\mathcal{N}_1$-type is a set $\vec p_W=\{\vec p_{i,e}\}_{i\leq d}$ such  where $\vec p_1=\{\vec p_{1,w}\}_{w\in W}$ is an $\mathcal{N}_1$-${W\choose 1}$-type and for $i>1$, $\vec p_i=\{\vec p_{i,e}\}_{e\in{W\choose r_i}\setminus \mathcal{R}_{e_i}(\vec p_1)}$ is a $\mathcal{N}_i$-$({W\choose r_i}\setminus\mathcal{R}_{r_i}(\vec p_1))$-type.

  For any point $\vec x_W$, we write $tp_{\mathcal{N}_d,\ldots,\mathcal{N}_1}(\vec x_W)$ for the type given by $(tp_{\mathcal{N}_d,\ldots,\mathcal{N}_1}(\vec x_W))_i=tp_{\mathcal{N}_i}(\vec x_W)$.
\end{definition}
This definition really is what we should expect: a $\mathcal{N}_d,\ldots,\mathcal{N}_1$-type $\vec p$ assigns, for each $i\leq d$ and each $r_i$-sub-tuple $e$ without repeated $\mathcal{N}_1$-types, an $\mathcal{N}_i$-type $\vec p_{i,e}$. The tuples with repeated types are omitted because those tuples concentrate on diagonals, and have to be handled differently.

This is precisely a description of an infinitesimal complex: we wish to consider a subset of $\Omega^W$ where, for each $e\in{W\choose r_i}$, we restrict ourselves to the set of $\vec x_W$ such that $\vec x_e\in\vec p_{i,e}$.

The additional subtlety is that when when we have a repeated tuple $\vec p_{1,w}=\vec p_{1,w'}$, we don't want to consider more complicated types containing more than one of them.  (This is a technical point involving how we handle repeated vertices in the proof of induced hypergraph counting, but for now, observe that if $\vec p_{1,w}=\vec p_{1,w'}$ then we should expect a type $\vec p_{2,\{2,2'\}}$ containing both to concentrate on a diagonal; since the diagonal has measure $0$, this means our type concentrates on a set of measure $0$, which obstructs our ability to prove a counting lemma.)

\subsection{Dense Types}

We have already noted that being a point of density is really a property of the type, not that point.  For completeness, we restate the definition in terms of types.  Recall that, for any $\mathcal{N}$-type $p$ and any integer $j$, $p(j)$ is a set in $\mathcal{N}_j$ approximating $p$.
\begin{definition}
  Let $\mathcal{N}$ be a nested system of neighborhoods.  Let $f:\Omega^W\rightarrow[0,1]$ be given.  For any $\mathcal{N}$-type $p$ such that each $p(j)$ has positive measure, we define
  \[f^j_{\mathcal{N}}(p)=\frac{1}{\mu(p(j))}\int_{p(j)}f(\vec x)\,d\mu\]
  and
  \[f^+_{\mathcal{N}}(p)=\lim_{j\rightarrow\infty}f^j_{\mathcal{N}}(p).\]

  We say an $\mathcal{N}$-$S$-type $p$ is a \emph{dense type for $f$} if $f^+_{\mathcal{N}}(p)$ exists and
  \[\lim_{j\rightarrow\infty}\frac{1}{\mu(p(j))}\int_{p(j)}|f^+_{\mathcal{N}}(y)-f^+_{\mathcal{N}}(p)|\,d\mu(y)=0.\]

  We say $p$ is a \emph{positive dense type for $f$} if $p$ is a dense type for $f$ and $f^+_{\mathcal{N}}(p)>0$.
\end{definition}

In order to prove hypergraph removal, we will need to consider types which are ``recursively'' dense types for $f$.  That means that when we have a $\vec p_e$, we need $\{\vec p_{e'}\}_{e'\subsetneq e}$ to be a dense type for each (or at least most) of the sets $\vec p_e(j)$. In order to make the inductive step work, we need to demand that $f$ be dense at $\vec p$ in a slightly stronger way.

We need to relativize the conditional expectation.  We take a $\sigma$-algebra $\mathcal{D}$, a function $f$, and a set $B$ which we should think of as being more complicated than those in $\mathcal{D}$ (for instance, we might have $\mathcal{D}=\mathcal{B}_{2,1}$ and $B\in\mathcal{B}_2\setminus\mathcal{B}_{2,1}$), and we want to define the conditional expectation of $f$ ``around the set $B$''.  We will write this $\mathbb{E}(f\curvearrowright B\mid\mathcal{D})$, which will be precisely the $\mathcal{D}$-measurable information with the property that, when given $B$, we can reconstruct $\mathbb{E}(f\chi_B\mid\mathcal{D})$.

\begin{definition}
  Let $f$ be a function, $P$ a set, and $\mathcal{D}$ a $\sigma$-algebra.  The \emph{weighted projection} $\mathbb{E}(f\curvearrowright P\mid\mathcal{D})$ is defined to be the unique (up to $L^2$-norm) function with domain $\{\vec x\mid\mathbb{E}(\chi_P\mid\mathcal{D})(\vec x)>0\}$ such that
\[\mathbb{E}(f\curvearrowright P\mid\mathcal{D})(\vec x)=\frac{\mathbb{E}(f\chi_P\mid\mathcal{D})(\vec x)}{\mathbb{E}(\chi_P\mid\mathcal{D})(\vec x)}.\]
\end{definition}
Note that $\mathbb{E}(f\curvearrowright P\mid\mathcal{D})$ is, as the notation suggests, measurable with respect to $\mathcal{D}$.  The main fact we will need is the following.

\begin{lemma}
  If $g$ is $\mathcal{D}$-measurable then
  \[\int f\chi_Pg\,d\mu=\int \mathbb{E}(f\curvearrowright P\mid\mathcal{D})\chi_Pg\,d\mu.\]
\end{lemma}
\begin{proof}
  \begin{align*}
    \int f\chi_Pg\,d\mu
    &=\int\mathbb{E}(f\chi_P\mid\mathcal{D}) g\,d\mu\\
    &=\int \mathbb{E}(f\curvearrowright P\mid\mathcal{D}) \mathbb{E}(\chi_P\mid\mathcal{D})g\,d\mu\\
    &=\int \mathbb{E}(f\curvearrowright P\mid\mathcal{D}) \chi_P g\,d\mu.
  \end{align*}
\end{proof}

\begin{definition}
  Given $f:\Omega^W\rightarrow[0,1]$, a set $P\subseteq \Omega^W$, and two systems of neighborhoods $\mathcal{N}_d,\mathcal{N}_{d-1}$, we define
  \[f^{+\curvearrowright P}_{\mathcal{N}_d,\mathcal{N}_{d-1}}=\mathbb{E}(\{\vec y\mid f^+_{\mathcal{N}_d}(\vec y)>0\}\curvearrowright P\mid\mathcal{K}_{W,r_{d-1}}(\mathcal{N}^\sigma_{d-1})\}).\]
\end{definition}
This obscure definition is justified by its crucial appearance in Lemma \ref{thm:hypergraph_counting} below.  In practice, $P$ will have the form $T_{{W\choose r_d}}(\{\vec p_e(j)\})$ for some $\mathcal{N}_d$-type $\vec p$, so we will have partitioned $\Omega^W$ into sets $P$ of this form and then we can think of the functions $f^{+\curvearrowright P}_{\mathcal{N}_d,\mathcal{N}_{d-1}}$ as being a ``partition of unity'' applied to the function $\mathbb{E}(\{y\mid f^+_{\mathcal{N}_d}(\vec y)>0\}\mid\mathcal{K}_{W,r_{d-1}}(\mathcal{N}^\sigma_{d-1})\})$.

\begin{definition}

  For each $i\leq d$, let $\mathcal{N}_i$ be a system of neighborhoods and let $f:\Omega^W\rightarrow\mathbb{R}$.  We say a $\mathcal{N}_d,\ldots,\mathcal{N}_1$-type $\vec p_W$ with $\mathcal{R}(\vec p_1)=\emptyset$ is a \emph{dense type for $f$ in $\mathcal{N}_d,\ldots,\mathcal{N}_1$} if:
\begin{itemize}
\item $\vec p_d$ is a dense type for $f$ (as an $\mathcal{N}_d$-type)
\item for all $j$ and each $e\in{W\choose r_d}$, $\{\vec p_{i,e'}\}_{i<d,e'\in{e\choose r_i}}$ is a dense type for every element of $\mathcal{N}_d(j)$,
\item for all $j$ and each $e\in{W\choose r_d}$, $\{\vec p_{i,e'}\}_{i<d,e'\in{e\choose r_i}}$ is a positive dense type for $\vec p_{d,e}(j)$,
\item for every $E$, for sufficiently large $j$, $\{\vec p_i\}_{i<d}$ is a positive dense type for
  \[\{\vec z\mid f^{+\curvearrowright T_{{W\choose r_d}}(\{\vec p_{d,e}(j)\})}_{\mathcal{N}_d,\mathcal{N}_{d-1}}(\vec z)>1-\frac{1}{E}\}\]
  in $\mathcal{N}_{d-1},\ldots,\mathcal{N}_1$.
\end{itemize}
If, additionally, $f^+_{\mathcal{N}_d}(\vec p)>0$, we say $\vec p$ is a positive dense type for $f$ in $\mathcal{N}_d,\ldots,\mathcal{N}_1$.

\end{definition}




\begin{lemma}
  For any measurable $f:\Omega^k\rightarrow[0,1]$ and almost every $x$, $tp_{\mathcal{N}_d,\ldots,\mathcal{N}_1}(x)$ is a dense type for $f$, and for almost every $x$ with $f(x)>0$, $tp_{\mathcal{N}_d,\ldots,\mathcal{N}_1}(x)$ is a positive dense type for $f$.
\end{lemma}
\begin{proof}
The set of $x$ so that $\mathcal{R}(tp_{\mathcal{N}_d,\ldots,\mathcal{N}_1}(x))\neq\emptyset$ has measure $0$, so we may ignore these points.
  
We now proceed by induction on $d$.  $tp(x)_{\mathcal{N}_d}$ is a dense type for $f$ in $\mathcal{N}_d$ exactly when $x$ is, and we have already shown that the set of $x$ such that $x$ is dense point for $f$ has measure $1$.

For each $P\in\bigcup_j\mathcal{N}_d(j)$, by the inductive hypothesis the set of $x$ such that $tp_{\mathcal{N}_{d-1},\ldots,\mathcal{N}_1}(x)$ is a dense type for $P$ has measure $1$.  Since there are countably many elements in $\bigcup_j\mathcal{N}_d(j)$, the set of of $x$ so that $tp_{\mathcal{N}_{d-1},\ldots,\mathcal{N}_1}(x)$ is a dense type for all of them also has measure $1$.

Also, for each $P\in\bigcup_j\mathcal{N}_d(j)$, the set of $x\in P$ such that $tp_{\mathcal{N}_{d-1},\ldots,\mathcal{N}_1}(x)$ is not a positive dense type for $P$ has measure $0$, and so again, except on a set of measure $0$, $tp_{\mathcal{N}_{d-1},\ldots,\mathcal{N}_1}(x)$ will be a positive dense type for $\mathcal{N}_d^j(x)$.

  It remains to show that, for each $\delta>0$ and each $E$, the set of points failing the fourth condition above with $E$ has measure $<\delta$.

  Let $\delta,E$ be given.  Let $A^+=\{\vec y\mid f^+_{\mathcal{N}_d}(\vec y)>0\}$.  By choosing $j$ sufficiently large, we can arrange that $A^+$ is contained, except for a set of measure $<\delta/2$, in elements $P\in\mathcal{K}_{W,r_d}(\mathcal{N}^\sigma_d(j))$ such that $\frac{\mu(A^+\cap P)}{\mu(P)}>1-\frac{\delta}{2E}$.

  Within any such $P$,
  \begin{align*}
    1-\frac{\delta}{2}
    &<\frac{1}{\mu(P)}\mu(A^+\cap P)\\
    &=\frac{1}{\mu(P)}\int \chi_{A^+}\chi_P\,d\mu\\
    &=\frac{1}{\mu(P)}\int \mathbb{E}(\chi_{A^+}\curvearrowright P\mid\mathcal{K}_{W,d-1}(\mathcal{N}^\sigma_{d-1}))\chi_P\,d\mu\\
    &=\frac{1}{\mu(P)}\int f^{+\curvearrowright P}_{\mathcal{N}_d,\mathcal{N}_{d-1}}\chi_P\,d\mu
  \end{align*}
  and therefore
  \[\frac{\mu(\{\vec y\in P\mid f^{+\curvearrowright P}_{\mathcal{N}_d,\mathcal{N}_{d-1}}(\vec y)>1-\frac{1}{E}\})}{\mu(P)}>1-\frac{\delta}{2}\]
  as well.

  Let $S_E=\{\vec y\mid f^{+\curvearrowright P}_{\mathcal{N}_d,\mathcal{N}_{d-1}}(\vec y)>1-\frac{1}{E}\}$.  Inductively, the set of $\vec x\in S_E\cap P$ such that $tp_{\mathcal{N}_{d-1},\ldots,\mathcal{N}_1}(\vec x)$ is not a positive dense type for $S_E$ has measure at most $\delta\mu(P)/2$, and therefore the set of $\vec x\in S_E$ such that $tp_{\mathcal{N}_{d-1},\ldots,\mathcal{N}_1}(\vec x)$ is not a positive dense type for $S_E$ has measure at most $\delta$.


\end{proof}

\subsection{Counting and Removal}

The next result is the analog of the hypergraph counting lemma.  We suppose we have a configuration $\{f_e\}_{e\in S}$ with $S$ a set of subsets of $W$, and we have points $\vec x_W$ with $f_e(\vec x_e)>0$ which is ``sufficiently generic'', in the sense that, for each $e$, $tp_{\mathcal{N}_d,\ldots,\mathcal{N}_1}(\vec x_e)$ is a positive dense type for $f_e$, then actually we can expand this single point into a set of points of positive measure, showing that $t_S(\{f_e\})>0$.

\begin{theorem}\label{thm:hypergraph_counting}
  Let $W$ be a finite set, let $\mathcal{N}_d,\ldots,\mathcal{N}_1$ be a properly aligned sequence of systems of neighborhoods so that $\mathcal{N}_i$ is a nested system of neighborhoods with arity $r_i$, let $S$ be a set of subsets of $W$, and suppose that $\vec p_W$ is a $\mathcal{N}_d,\ldots,\mathcal{N}_1$-type such that:
  \begin{itemize}
  \item for each $e\in S$, the restriction $\vec p_W$ is a positive dense type for $f_e$, and
  \item for each $e\in S$, either:
    \begin{itemize}
    \item $f_e$ is $\mathcal{K}_{e,r_d}(\mathcal{N}^\sigma_d)$-measurable, or
    \item for every $e'\in S\setminus\{e\}$, the function $\vec x_e\mapsto \int f_{e'}(\vec x_{e\cup e'})\,d\mu(\vec x_{e'\setminus e})$ is $\mathcal{K}_{e,r_d}(\mathcal{N}^\sigma_d)$-measurable.
    \end{itemize}
  \end{itemize}
  Then $t_S(\{f_e\})>0$.
  \end{theorem}
\begin{proof}
  We proceed by induction on $d$.  When $d=1$, this is exactly Theorem \ref{thm:graph_counting}.

So suppose $d>1$.  By Lemma \ref{thm:reduction} with $\mathcal{D}_e=\mathcal{K}_{e,r_d}(\mathcal{N}_d^\sigma)$ for all $e$, we have $t_S(\{f_e\})=t_S(\{(f_e)^+_{\mathcal{N}_d}\})$.  Since $\vec p_{e}$ is a positive dense type of $f_e$, also $(f_e)^+_{\mathcal{N}_d}(\vec p_{e})>0$ for each $e$.  Let $A^+_e=\{\vec y_e\mid (f_e)^+_{\mathcal{N}_d}(\vec y_e)>0\}$.  It suffices to show that $t_S(\{\chi_{A^+_e}\})>0$.

Choose $j$ sufficiently large.  For each $e\in S$, let
\[A^\flat_e=\{\vec z\mid (\chi_{A^+_e})^{+\curvearrowright T_{{e\choose r_d}}(\{\vec p_{d,e'}(j)\})}_{\mathcal{N}_d,\mathcal{N}_{d-1}}(\vec z)>1-\frac{1}{|S|+1}\},\]
so $\{\vec p_{i,e}\}_{i<d}$ is a positive dense type for $A^\flat_e$ in $\mathcal{N}_{d-1},\ldots,\mathcal{N}_1$.  Let $S'=S\uplus{W\choose r_d}$; for $e\in S$, we let $g_e=\chi_{A^\flat_e}$, and for $e\in{W\choose r_d}$, we let $g_e=\chi_{\vec p_{d,e}(j)}$.

Then, for $e\in S$, $g_e$ is $\mathcal{K}_{e,r_{d-1}}(\mathcal{N}^\sigma_{d-1})$-measurable.  Therefore, since $\mathcal{N}_d$ is properly situated in $\mathcal{N}_{d-1}$, the inductive hypothesis applies, so $t_{S'}(\{g_e\})>0$.

Therefore 
\[\mu(T_S(\{A^\flat_e\})\cap T_{{W\choose r_d}}(\{\vec p_{d,e}(j))\})>0.\]
Consider some $\vec y_W\in \left[T_S(\{A^\flat_e\})\setminus T_S(\{A^+_e\})\right]\cap T_{{W\choose r_d}}(\vec p_{d,e}(j))$.  There must be some $e_0\in S$ such that $\vec y_{e_0}\in A^\flat_{e_0}\setminus A^+_{e_0}$.  For each $e_0\in S$,
\begin{align*}
  &\mu(\{\vec y_W\in T_{{W\choose r_d}}(\vec p_{d,e}(j))\mid \vec y_{e_0}\in T_S(\{A^\flat_e\})\setminus A^+_{e_0}\})\\
  =&\int \chi_{A^\flat_{e_0}}(1-\chi_{A^+_{e_0}})\prod_{e\in{W\choose r_d}}\chi_{\vec p_{d,e}(j)}\cdot\prod_{e\in S\setminus\{e_0\}}\chi_{A^\flat_e}\,d\mu\\
  =&\int \chi_{A^\flat_{e_0}}(1-\chi_{A^+_{e_0}})\prod_{e'\in{e_0\choose r_d}}\chi_{\vec p_{d,e'}(j)}\cdot\prod_{e\in S\setminus\{e_0\}}\chi_{A^\flat_e}\prod_{e\in{W\choose r_d}\setminus{e_0\choose r_d}}\chi_{\vec p_{d,e}(j)}\,d\mu\\
  =&\int \chi_{A^\flat_{e_0}}(1-(\chi_{A^+_e})^{+\curvearrowright T_{{e_0\choose r_d}}(\{\vec p_{d,e'}\})}) \prod_{e'\in{e_0\choose r_d}}\chi_{\vec p_{d,e'}(j)}\cdot\prod_{e\in S\setminus\{e_0\}}\chi_{A^\flat_e}\prod_{e\in{W\choose r_d}\setminus{e_0\choose r_d}}\chi_{\vec p_{d,e}(j)}\,d\mu\\
  &<\frac{1}{|S|+1}\mu(T_S(\{A^\flat_e\})\cap T_{{W\choose r_d}}(\{\vec p_{d,e}(j)\})).
\end{align*}
Therefore
\[\mu(\left[T_S(\{A^\flat_e\})\setminus T_S(\{A^+_e\})\right]\cap T_{{W\choose r_d}}(\{\vec p_{d,e}(j)\})))<\frac{|S|}{|S|+1}\mu(T_S(\{A^\flat_e\})\cap T_{{W\choose r_d}}(\{\vec p_{d,e}(j)\}))),\]
so
\[t_S(\{A^+_e\})\geq\mu(T_S(\{A^\flat_e\})\cap \vec p_d(j))>\frac{1}{|S|+1}\mu(T_S(\{A^\flat_e\})\cap T_{{W\choose r_d}}(\{\vec p_{d,e}(j)\})))>0.\]
\end{proof}

\begin{theorem}
  Suppose $H=(W,F)$ is a finite $k$-graph and $G=(\Omega,E)$ is a $k$-graph with a countably approximated atomless Keisler graded probability space on $\Omega$ with $E\in\mathcal{B}^0_k$.  If $t_H(G)=0$ then there is a symmetric $E'\subseteq E$ such that $E\setminus E'$ is a measure $0$ set contained in an intersection of sets in $\mathcal{B}^0_k$ and, taking $G'=(\Omega,E')$, $T_H(G')=\emptyset$.
\end{theorem}
\begin{proof}
  Nearly identical to the proof of Graph Removal, Theorem \ref{thm:graph_removal}.  Let $\mathcal{N}_{k-1},\ldots,\mathcal{N}_1$ be a sequence of systems of neighborhoods for the main example.  Then let $E'\subseteq E$ consist of the points in $E$ whose type is positive dense for $\chi_E$.  If $T_H(E')\neq\emptyset$ then any $\vec x_W\in T_H(E')$ satisfies the conditions of the previous lemma, and therefore $t_H(E)>0$.
\end{proof}

\begin{cor}
    For every finite $k$-graph $H=(W,F)$ and every $\epsilon>0$ there is a $\delta>0$ so that whenever $G=(V,E)$ is a $k$-graph with $t_H(G)<\delta$, there is a symmetric $E'\subseteq E$ with $|E\setminus E'|<\epsilon|V|^k$ such that, taking $G'=(V,E')$, $T_H(G')=\emptyset$.
\end{cor}

\subsection{Conditioning on Sets of Measure $0$}

Before going on, it will be convenient to consider the notion of picking a type ``uniformly at random''.  The natural way to pick a random type is to pick a random point $\vec x$ and take $tp_{\mathcal{N}_d,\ldots,\mathcal{N}_1}(\vec x)$.  (Note that, with probability $1$, such a type has no repeated $\mathcal{N}_1$-types, so we do not need to worry about that complication here.)

However what we will need later is to first pick $\mathcal{N}_1$-types, then the $\mathcal{N}_2$-type, and so on, and we will need to describe what it means to pick a $\mathcal{N}_2$-type randomly among the extensions of a given $\mathcal{N}_1$-type.

Because the types represent sets of measure $0$, we do not generally expect to be able to make sense of the probability of an event conditioned on being in a type $\vec p$.  However because these events are intersections of a well-defined family of positive measure events, we can make sense of conditioning on them as long as the right limits converge.

Say we have two systems of neighborhoods, $\mathcal{N}$ and $\mathcal{M}$ of arity $r\leq s$, respectively.  (For instance, $\mathcal{M}=\mathcal{N}_{i+1}$ while $\mathcal{N}=\mathcal{N}_i$.)  Then almost every $\mathcal{N}$-$s$-type $\vec p$ is a dense type for every $P$ in every $\mathcal{M}(c)$, since there are only countably many such sets.



For instance (in the standard example) choosing $\mathcal{N}_1$ types $p$ and $q$ gives us a measure $0$ rectangle $p\times q$; despite being measure $0$, for almost all $p$ and $q$ we can make sense of choosing a pair $(x,y)\in p\times q$ randomly and taking $tp_{\mathcal{N}_2}(x,y)$: the probability that $tp_{\mathcal{N}_2}(x,y)(j)=P$ is precisely the density of $P$ in the type $p\times q$.

So we may choose a $\mathcal{N}_d,\ldots,\mathcal{N}_1$-type by first choosing $\mathcal{N}_1$-types randomly, and then inductively using this process to choose the $\mathcal{N}_2$-types, then the $\mathcal{N}_3$-types, and so on.

The only thing we need to check is that this gives the same distribution as if we had simply chosen the type of a random point.

\begin{lemma}
The inductive method of choosing $\mathcal{N}_d,\ldots,\mathcal{N}_1$-types has the same distribution as choosing $tp_{\mathcal{N}_d,\ldots,\mathcal{N}_1}(\vec x)$ for a uniformly chosen $\vec x$.
\end{lemma}
\begin{proof}
  By induction on $i\leq d$.  For $i=1$ these distributions have the same definition.  Suppose the claim holds for $i$.  The probability that we choose a $\mathcal{N}_{i+1}$-type $\vec p_{i+1}$ with $\vec p_{i+1}(j)=P$ is the integral, over all choices of $\vec p_i$ of $(\chi_P)^+_{\mathcal{N}_i}(\vec p_i)$.  By the inductive hypothesis, this is the integral over a random $\vec x$ of $(\chi_P)^+_{\mathcal{N}_i}(tp_{\mathcal{N}_i}(\vec x))$, which is equal to $\mu(P)$.
\end{proof}

\subsection{Induced Hypergraph Removal}

  

To prove induced hypergraph removal, we need a hypergraph counting lemma that allows repeated elements in tuples.  To do that, we need to generalize the notion of a likely configuration.

Suppose we have a tuple of points $\vec a_W$ where some of the points are repeated\footnote{Really, this should be talking about types rather than points, but we can more or less equate $a_w$ with $tp_{\mathcal{N}_1}(a_w)$, and this will be clearer without the added abstraction of talking about types.}.  Once again, we want to be able to ``wiggle the points'' so that we can replace them with nearby points which we will be able to apply Theorem \ref{thm:hypergraph_counting} to.  The complication is that now, in addition to the types of the points themselves, we need to be concerned with the types of the tuples they belong to: if we ``wiggle'' $a_1$, this also affects the neighborhood of $tp_{\mathcal{N}_2}(a_1,a_2)$.  Indeed, we can't really ``wiggle'' $a_1$ while holding $tp_{\mathcal{N}_2}(a_1,a_2)$ constant, because $tp_{\mathcal{N}_2}(a_1,a_2)$ completely determines $tp_{\mathcal{N}_1}(a_1)$.

So we have to wait until later in the counting process: the proof of Theorem \ref{thm:hypergraph_counting} inductively reduces a hypergraph counting problem to a problem about counting graphs.  In particular, prior to the last step of that process, we replace $tp_{\mathcal{N}_2}(a_1,a_2)$ with a positive measure approximation to it.  (Specifically, the set $A^\flat_{1,2}$ which we construct in that proof.)  At \emph{that} point, we can safely ``wiggle'' $a_1$, since we can just promise to remain within various sets which have positive measure.

With that in mind, we can prove our infinitary version of removal.  We state it in a general form, allowing a coloring $\rho$ of $k$-tuples and showing that, with a small change to $\rho$, we can simultaneously remove all copies of small colorings $(W,c)$ which appear with $0$ density in $(\Omega,\rho)$.  As always, the induced hypergraph case is when $\Sigma=\{0,1\}$.

\begin{theorem}\label{thm:induced_hypergraph_removal}
  Let $\Sigma$ be a finite set and let $(\Omega,\rho)$ be a coloring with a countably approximated atomless Keisler graded probability space on $\Omega$ and let $\rho:{\Omega\choose k}\rightarrow\Sigma$ be such that each $\rho^{-1}(\sigma)\in\mathcal{B}^0_2$.  For each $\epsilon>0$ there is a $\rho':{\Omega\choose k}\rightarrow\Sigma$ such that
  \[\mu(\{\vec x\mid\rho(\vec x)\neq \rho'(\vec x)\})<\epsilon,\]
  each $(\rho')^{-1}(\sigma)\in\mathcal{B}^0_2$, and, for all $(W,c)$, if $t_{W,c}(\Omega,\rho)=0$ then $T_{W,c}(\Omega,\rho')=\emptyset$.
\end{theorem}
\begin{proof}
To emphasize the way the proof generalizes, we will state it in terms of a general sequence of systems of neighborhoods, $\mathcal{N}_{d-1},\ldots,\mathcal{N}_1$.  For this precise result, we can take $d=k$ and $r_i=i$ for all $i<d$, but nothing changes to consider a longer sequence of systems of neighborhoods, and we will need this case in the next section.
  
  \textbf{General Setup}:
  For each $\sigma\in\Sigma$, let $P_\sigma=\{\vec y\mid \rho(\vec y)=\sigma\}$ and $f_\sigma=\chi_{P_\sigma}$, so $T_{W,c}(\Omega,\rho)=T_{{W\choose k}}(\{P_{\sigma(e)}\}_{e\in{W\choose k}})$.  Let $\mathcal{P}=\{P_\sigma\mid\sigma\in\Sigma\}$, which is a partition of $k$-tuples from $\Omega$.  Choose a sequence $c_{d-1}<\cdots<c_1$ where each $c_j$ is sufficiently large relative to the sizes of $|\mathcal{N}_{j'}^{c_{j'}}|$ for $j'>j$ and so $c_1$ is large enough that the set of $k$-tuples with more than one point in the same element of $\mathcal{N}_1^{c_1}$ has measure $<\epsilon/3$.

  We will sometimes think of $\mathcal{P}$ as being analogous to $\mathcal{N}_d$; in particular, we define $r_d=k$.

  Our plan is this.  We have partitioned $\Omega^{r_{d-1}}$ into the elements of $\mathcal{N}_{d-1}^{c_{d-1}}$, then partitioned $\Omega^{r_{d-2}}$ into the elements of $\mathcal{N}_{d-2}^{c_{d-2}}$, which are much smaller, and so on.  By analogy to the proof of Theorem \ref{thm:induced_removal}, we will want to consider the partition of $\Omega^k$ into sets of the form $T_{\biguplus_{i<d}{[k]\choose r_i}}$ where, for each $e\in{[k]\choose r_i}$, we have $P_e\in\mathcal{N}_{i}^{c_{i}}$.

  What we would like to do is choose, from each of the components $\vec P=\{P_e\}$, a single element $\vec x^{\vec P}\in T_{\biguplus_{i<d}{[k]\choose r_i}}(\{P_e\})$ so that $tp_{\mathcal{N}_{k-1},\ldots,\mathcal{N}_1}(\vec x)$ is a dense type for each $P_\sigma$.  Then we will define $\rho'$ so that, on the set $T_{\biguplus_{i<d}{[k]\choose r_i}}(\{P_e\})$, $\rho'$ only takes values $\sigma$ such that $\chi_{P_\sigma}^+(\vec x^{\vec P})>0$.

  There are two complications.  In the proof of Theorem \ref{thm:induced_removal}, we had to deal separately with components of the form $P\times P$.  Here, analogously, we have to deal separately with the case where $\vec P_{\{i\}}=\vec P_{\{j\}}$ for $i\neq j$---that is, the case where two of the singleton components of $\vec P$ are the same element of $\mathcal{N}_1^{c_1}$.  As in Theorem \ref{thm:induced_removal}, these components account for a small amount of measure, so we have a great deal of freedom in how we define $\rho'$ on them.  We address this later, after dealing with the other components.

  The second complication is that when we choose $\vec x^{\vec P}$, we need the choices to be ``coherent'': if we have two components $\vec P$ and $\vec Q$ such that $\vec P_{\{i\}}=\vec Q_{\{j\}}$, we need $tp_{\mathcal{N}_1}(\vec x^{\vec P}_i)=tp_{\mathcal{N}_1}(\vec x^{\vec Q}_j)$.  More generally, if there is an $e_0$ such that $P_e=Q_e$ for $e\subseteq e_0$, we need $tp_{\mathcal{N}_{|e|}}(\vec x^{\vec P}_e)=tp_{\mathcal{N}_{|e|}}(\vec x^{\vec Q}_e)$ for $e\subseteq e_0$.  (Really, we need something slightly more general: if there is a bijection $\pi:[k]\rightarrow[k]$ and an $e_0$ such that $P_e=Q_{\pi(e)}$ for $e\subseteq e_0$ then we should have $tp_{\mathcal{N}_{|e|}}(\vec x^{\vec P}_e)=tp_{\mathcal{N}_{|e|}}(\vec x^{\vec Q}_{\pi(e)})$.)

  This is why we introduced the alternate method of selecting types in the previous section.  For each $P\in\mathcal{N}_1^{c_1}$ with positive measure, we will choose a type $p^P\subseteq P$.  (That is, so that $p^P(c_1)=P$.)  Then we will turn to pairs: given a triple $\vec P=\{P_{1,2},P_1,P_2\}$ we will take $\vec p^{\vec P}_1=p^{P_1}$, $\vec p^{\vec P}_2=p^{P_2}$, and as long as $P_{1,2}$ has positive density at $p^{P_1}\times p^{P_2}$, we will choose a $p^{\vec P}_{1,2}\subseteq P^{1,2}\cap(p^{P_1}\times p^{P_2})$.

  \textbf{Configurations}:
  Let us make all this precise.  As the discussion above suggests, we will want to work inductively, starting with the partition of $\Omega$ and working our way up to the partition of $\Omega^k$.  We will call the components of these partitions \emph{configurations}.  For $j\leq d$, 
  let us define a \emph{${\leq}j$-configuration} to be a collection $\{P_e\}_{e\in\biguplus_{[j']\leq \max\{j,d-1\}}{[r_j]\choose r_{j'}}}$ such that for each $e\in\biguplus_{j'\leq \max\{j,d-1\}}{[r_j]\choose r_{j'}}$, $P_e\in\mathcal{N}_{|e|}^{c_{|e|}}$.  (Note that ${\leq}d$-configurations are slightly disanalogous, since they do not have a ``top level'' component from $\mathcal{P}$.)

It will be convenient to restrict ${\leq}j$-configurations in the natural way: if $\vec P=\{P_e\}_{e\in\biguplus_{j'\leq \max\{j,d-1\}}{[r_j]\choose r_{j'}}}$ is a ${\leq}j$-configuration, $j_0<j$, and $e_0\in{[j]\choose j_0}$, we can define $\vec P\upharpoonright e_0$ to be the ${\leq}j'$-configuration $\{P_e\}_{e\in\biguplus_{j'\leq j_0}{[e_0]\choose r_{j'}}}$.

  For each $j\leq d$, 
  the ${\leq} j$-configurations partition $\Omega^{r_j}$ into sets of the form $T_{\biguplus_{j'\leq \max\{j,d-1\}}{[r_j]\choose r_{j'}}} (\{P_e\}_{e\in\biguplus_{j'\leq \max\{j,d-1\}}{[r_j]\choose r_{j'}}})$.

  We say a configuration $\{P_e\}$ \emph{has distinct singletons} if whenever $i,i'\in[j]$ with $i\neq i'$, $P_{\{i\}}\neq P_{\{i'\}}$.  We will first deal with the configurations with distinct singletons, and then deal with the remaining configurations.  Note that the configurations without distinct singletons account for a small amount of measure.

  We ultimately want to associate each ${\leq}j$-configuration $\vec P$ with a type $\vec p^{\vec P}$, which we do by induction on $j$.  However there is a technical issue we must address first.  If $\mu(T_{j'\leq \max\{j,d-1\}}(\{P_e\}))=0$ then we cannot choose a random type refining this configurations.  Slightly more generally, if $j>1$ then our choice of types for $j-1$-tuples may lead to $P_{[j]}$ having density $0$ at the corresponding type of $j$-tuples, which will also lead to us being unable to continue the process.  In both these cases, we will consider $\vec P$ a \emph{defective} configuration.  We therefore assign, to each configuration $\vec P$, an associated configuration $\vec Q^{\vec P}$, which is always non-defective.  In most cases, we will have $\vec Q^{\vec P}=\vec P$, but when $\vec P$ is defective, $\vec Q^{\vec P}$ will be a different configuration.  (As this name suggests, only a small amount of measure will be contained in defective configurations.)  We will then choose the type $\vec p^{\vec P}$ to refine $\vec Q^{\vec P}$.

  \textbf{Choosing representative types}:
  By induction on $j$, for each $j<d$ and each ${\leq}d$-configuration $\vec P=\{P_e\}$ with distinct singletons we will choose a ${\leq}j$-configuration $\vec Q^{\vec P}$ and a $\mathcal{N}_j,\ldots,\mathcal{N}_1$-type $\vec p^{\vec P}$.

  Once we have completed this definition for $j$, it is natural to define $\mathcal{N}_j,\ldots,\mathcal{N}_1$-$[r_{j'}]$-types for any ${\leq}j'$ configuration: when $\vec P$ is a ${\leq}j'$-configuration with $j'>j$, we define $\vec p^{\vec P\upharpoonright j}$ to be the type with $\vec p^{\vec P\upharpoonright j}_e=\vec p^{\vec P\upharpoonright e}_e$ for all $e\in\biguplus_{j_0\leq j}{[j']\choose j_0}$.

  We inductively arrange that:
  \begin{enumerate}
  \item the choices are cumulative: for $j'<j$ and $e_0\in{[j]\choose r_{j'}}$, $\vec Q^{\vec P}_{e_0}=\vec Q^{\vec P\upharpoonright e_0}_{e_0}$ and $\vec p^{\vec P}_{e_0}=\vec p^{\vec P\upharpoonright e_0}_{e_0}$,
  \item the choices are symmetric: if $\pi:[j]\rightarrow[j]$ is a permutation and $\vec P$, $(\vec Q^{\vec P^\pi})^{\pi^{-1}}=\vec Q^{\vec P}$ and $(\vec p^{\vec P^\pi})^{\pi^{-1}}=\vec p^{\vec P}$,
  \item the types $\vec p^{\vec P}$ refine the configurations: for each $\vec P$, each $j'\leq j$, and each $e\in{[j]\choose r_{j'}}$, $\vec p^{\vec P}(c_{j'})=\vec Q^{\vec P}_e$,
  \item $\vec p^{\vec P}$ is a positive dense type for $\vec Q^{\vec P}_{[j]}$,
  \item for every $j'\in(j,d]$ and every ${\leq}j'$-configuration $\vec P$ with distinct singletons, the type $\vec p^{\vec P\upharpoonright j}$ is a dense type for $\vec P_{[j']}$,
  \item few points belong to configurations represented by types which make the sets in the configuration very sparse: the set of points belonging to ${\leq }d$-configurations $\vec P$ with distinct singletons such that, for some $e_0\in\biguplus_{j'\in(j,d)}{[d]\choose r_{j'}}$, the density of $P_{e_0}$ at the $\mathcal{N}_j,\ldots,\mathcal{N}_1$-$e_0$-type $\vec p^{P\upharpoonright e_0}$ is $<\frac{\epsilon}{3k}$, has measure at most $\frac{\epsilon}{3}$.
  \end{enumerate}
We will describe a random construction of the entire sequence and then argue that, with positive probability, we can find a choice satisfying these conditions.
  
  For $j=1$, a ${\leq}1$-configuration is just a set $P\in\mathcal{N}_1^{c_1}$.  If $\mu(P)=0$, we take $Q^P$ to be some set in $\mathcal{N}_1^{c_1}$ with positive measure, otherwise we take $Q^P=P$.  We then take $\vec p^{P}$ to be $tp_{\mathcal{N}_1}(x)$ for a randomly chosen element $x\in Q^P$.  

  Suppose we have completed the construction for $j$.  Consider the equivalence classes consisting of ${\leq}j+1$-configurations $\vec P$ where we identify configurations under permutations of $[j+1]$; we will choose a single representative from each such equivalence class.

  We can choose $\vec Q^{\vec P}$ and $\vec p^{\vec P}$ as follows.  We first consider the non-defective configurations.  Consider a ${\leq}j+1$-configuration $\vec P$ with distinct singletons such that:
  \begin{enumerate}
  \item for each $e\subsetneq[r_{j+1}]$, we have $\vec Q^{\vec P}_e=\vec P_e$, and
  \item $\chi_{\vec P_{[r_{j+1}]}}^+(\{\vec p^{\vec P}_e\}_{e\subsetneq [r_{j+1}]})>0$.
  \end{enumerate}
  (Note that we will arrange to have $\chi_{\vec P_{[r_{j+1}]}}^+(\{\vec p^{\vec P}_e\}_{e\subsetneq [r_{j+1}]})$ exist because we will arrange for (5) to hold).  Then we set $\vec Q^{\vec P}_{[r_{j+1}]}=\vec P_{[r_{j+1}]}$ and choose $\vec p^{\vec P}_{[r_{j+1}]}$ to be a random refinement of $\vec Q^{\vec P}_{[r_{j+1}]}$ at $\{\vec p^{\vec P}_e\}_{e\subsetneq [r_{j+1}]}$, as described in the previous subsection.  (The choice of the $\vec p^{\vec P}$ is the only non-deterministic part of the construction.)


  Consider a defective configuration $\vec P$ such that for each $e\subsetneq [r_{j+1}]$, we have $\vec Q^{\vec P}_e=\vec P_e$ but $\chi_{\vec P_{[r_{j+1}]}}^+(\{\vec p^{\vec P}_e\}_{e\subsetneq [r_{j+1}]})=0$.  (That is, $\vec P\upharpoonright j$ was fine, but the top level component $\vec P_{[r_{j+1}]}$ makes it defective defective.)  Since $\{\vec p^{\vec P}_e\}_{e\subsetneq [r_{j+1}]}$ is a dense type for all $P\in\mathcal{N}_{j+1}^{c_{j+1}}$ and these sets are a partition, there is some $P\in\mathcal{N}_{j+1}^{c_{j+1}}$ with $\chi_P^+(\{\vec p^{\vec P}_e\}_{e\subsetneq [r_{j+1}]})>0$, and we take $\vec Q^{\vec P}_{[r_{j+1}]}=P$ and $\vec p^{\vec P}_{[r_{j+1}]}=\vec p^{\vec Q^{\vec P}}_{[r_{j+1}]}$.

  Finally, consider a defective configuration such that, for some $e\subsetneq [r_{j+1}]$, $\vec Q^{\vec P}_e\neq \vec P_e$.  Then we wish to simply follow along with the ``corrected'' configuration: define $\vec P'$ by $\vec P'_e=\vec Q^{\vec P}_e$ for $e\subsetneq[j+1]$ and $\vec P'_{[r_{j+1}]}=\vec P_{[r_{j+1}]}$, and set $\vec Q^{\vec P}_{[r_{j+1}]}=\vec Q^{\vec P'}_{[r_{j+1}]}$ (which was already defined in one of the previous cases) and $\vec p^{\vec P}_{[r_{j+1}]}=\vec p^{\vec P'}_{[r_{j+1}]}$.

  To satisfy symmetry, we define $\vec Q$ and $\vec p$ for permutations of $\vec P$ in the unique way determined by symmetry.  Note that we need to use the fact that $\vec P$ has distinct singletons to make sure that no permutation other than the identity maps $\vec P$ to itself, so the symmetry requirement imposes no further restrictions on our choices.

  We need to check that, with positive probability, the choice of the $\vec p^{\vec P}$ satisfies the six conditions above.  The first three follow immediately from the construction.

  The fourth property and fifth properties hold with probability $1$, so we can certainly choose the types $\vec p^{\vec P}$ to satisfy these properties.

  For the sixth property, note that the ${\leq} k$-configurations $\vec P$ such that there exists an $e_0\subseteq k$ and a $j<|e_0|$ such that the density of $P_{e_0}$ in $T_{\biguplus_{j'}{[k]\choose j'}}(\vec P)$ is $<\frac{\epsilon}{3k}$ have measure at most $\frac{\epsilon}{3}$.  Consider a ${\leq} k$-configuration such that this does not happen.  As we observed in the previous section, the corresponding $\mathcal{N}_j,\ldots,\mathcal{N}_1$-$e_0$-type $\vec p^{\vec P}$ is chosen with the same distribution as choosing the type of a random point in $T_{\biguplus_{j'}{[k]\choose j'}} (\vec P)$.  By our choice of $c_j$, for each ${\leq}k$-configuration $\vec P$, the probability that there is any $e_0,j$ so that $P_{e_0}$ is has density $0$ in $\vec p^{\vec P\upharpoonright j}$ is at most $\epsilon/3$.  By averaging over all $\leq k$-configurations (weighted by their size), there is positive probability we choose the $\vec p^{\vec P}$ that the set of $\leq k$-configurations failing the condition in (6) has measure at most $\frac{\epsilon}{3}$.  

  This last condition implies that most points belong to non-defective configurations: the only way there is an $e$ with $\vec P_e\neq \vec Q^{\vec P}_e$ is if there is an $e$ so that $P_e$ has density $0$ in the corresponding type of lower arity, which means all such configurations are contained in the set of exceptional configurations.

\textbf{Defining $\rho'$ for most tuples}:
At this point, we have done enough to define $\rho'$ on ${\leq k}$-configurations $\vec P$ with distinct singletons.

Let $\Xi$ be the set of non-empty subsets of $\Sigma$, and for each ${\leq} k$-configuration $\vec P$ with distinct singletons, let $\xi(\vec P)=\{\sigma\in\Sigma\mid (f_\sigma)^+_{\mathcal{N}_{k-1}}(\vec p^{\vec P})>0\}$.  (Since the $f_\sigma$ add to $1$, $\xi(\vec P)$ is always non-empty, and therefore in $\Xi$.)  We will define $\rho'$ on $T_{\biguplus_{j'}{[k]\choose j'}}(\vec P)$ by setting
  \[\rho'(\vec x)=\left\{\begin{array}{ll}
                      \rho(\vec x)&\text{if }\rho(\vec x)\in\xi(\vec P)\\
                      \text{some }\sigma\in\xi(\vec P)&\text{otherwise}.
                    \end{array}\right.\]

  Note that if $\rho(\vec x)\neq\rho'(\vec x)$, we must have one of:
  \begin{itemize}
  \item the ${\leq}k$-configuration containing $\vec x$ does not have distinct singletons,
  \item the ${\leq}k$-configuration containing $\vec x$ is defective, or
  \item $\rho(x)\not\in\xi(\vec P)$.
  \end{itemize}

  The first two conditions account for points of measure at most $\epsilon/3$ each.  If $\rho(x)\not\in\xi(\vec P)$ then we have $(f_{\rho(x)})^+_{\mathcal{N}_{k-1}}(\vec p^{\vec P})=0$.  Since the $\vec p^{\vec P}$ are distributed uniformly at random, except on a set of configurations of measure at most $\epsilon/6$, $(f_{\rho(x)})^+_{\mathcal{N}_{k-1}}(\vec p^{\vec P})=0$ implies that the set of points in $\vec P$ with color $\rho(x)$ has measure at most $(\epsilon/6)\mu(T_{\biguplus_{j'}{[k]\choose j'}}(\vec P))$.  Therefore (regardless of how we define $\rho'$ on the ${\leq}k$-configurations which do not have distinct singletons),
  \[\mu(\{\vec x\mid\rho(\vec x)\neq\rho'(\vec x)\})<\epsilon.\]

  \textbf{Dealing with tuples with distinct singletons}:
  Next, again as in Lemma \ref{thm:induced_removal}, we need to decide what to do with the configurations which do not have distinct singletons.

  To motivate the construction, it is useful to look at how we will use our definition.  Suppose that, after finishing the definition of $\rho'$, we have some $\vec x_W\in T_{W,c}(\Omega,\rho')$.  Then $\vec x_W$ induces some maps into our partition: for $j< d$ we can take $\theta_j:{W\choose r_j}\rightarrow\mathcal{N}_j^{c_j}$ given by $\theta_j(e)=\mathcal{N}_j^{c_j}(\vec x_e)$.  Then for $j\leq d$, whenever $e\in{W\choose r_j}$ and $\theta_1$ is injective on $e$, we can define $\Theta(e)$ to be the ${\leq}d$-configuration $\{\theta_j(e')\}_{e'\in\bigcup_{j'<d}{e\choose r_{j'}}}$.  When $e\in{W\choose k}$ and $\theta_1$ is injective on $e$, we must have $c(e)\in\xi(\Theta(e))$.

  Let us isolate this definition: a \emph{strict blow up of the partition} is a tuple $(W,\zeta,\{\theta_j\})$ where:
  \begin{itemize}
  \item each $\theta_j:{W\choose j}\rightarrow\mathcal{N}_j^{c_j}$,
  \item whenever $e\in{W\choose k}$ and $\theta_1$ is injective on $e$, $\zeta(e)$ is defined and equal to $\xi(\Theta(e))$.
  \end{itemize}
  
  Whenever we have a strict blow up $(W,\zeta,\{\theta_j\})$, we can take the type $\vec q_W$ where, for each $e\in \bigcup_{j<d}{W\choose r_j}$, $\vec q_e=\vec p^{\Theta(e)}_e$.  Theorem \ref{thm:hypergraph_counting} applies to $\vec q_W$, so $t_{W,\zeta}(\Omega,\xi)>0$.

  We need to consider how the remaining tuples in ${W\choose k}$ are mapped.  Let us consider tuples $(W,\zeta,\iota,\{\theta_j\})$ where $(W,\zeta,\{\theta_j\})$ is a strict blow up and $\iota:({W\choose k}\setminus\dom(\zeta))\rightarrow\Sigma$.  (We make the choice here to have $\zeta$ take values in $\Xi$ while $\iota$ only takes values in $\Sigma$; it would cause no harm, except perhaps additional complication, to instead let $\iota$ be $\Xi$-valued as well.)

  We want to consider tuples $(W,\zeta,\iota,\{\theta_j\})$ where $\iota$ is ``homogeneous'', in the sense that $\iota(e)$ only depends on the configuration $e$ is mapped to.  To make this precise, let us say $\vec P=\{P_e\}$ is a \emph{${\leq}k$-configuration with repeated singletons} if each $P_e\in\mathcal{N}_{|e|}^{c_e}$ and $P_e$ is defined for all $e$ such that, for $i,i'\in e$, $P_i\neq P_{i'}$.  (That is, when $e$ contains repeated elements of $\mathcal{N}_1^{c_1}$, we simply do not define $P_e$.)  Let us define $\mathcal{Z}$ to be the set of ${\leq}k$-configurations with repeated singletons.

  We can extend the definition of $\Theta$ to those $e\in{W\choose r_j}$ where $\theta_1$ is not injective on $e$ by defining $\Theta(e)\in\mathcal{Z}$ to be $\{\theta_j(e')\}_{e'\in\bigcup_{j'<d}{[r_j]\choose r_{j'}}\text{ and }\theta_1\text{ is injective on }e'}$.  When $\nu:\mathcal{Z}\rightarrow\Sigma$, let us say $(W,\zeta,\iota,\{\theta_j\})$ is \emph{$\nu$-homogeneous} if, for all $e\in{W\choose k}\setminus\dom(\zeta)$, $\iota(e)=\nu(\Theta(e))$.

  Given $\vec x_W\in\Omega^W$, we can of course induce a function $\iota:{W\choose k}\rightarrow\Sigma$ by taking $\iota(e)=\rho(\vec x_e)$.  So what we need to do is find such $\vec x_W$ which are homogeneous.

  Let us say $(W,\zeta,\{\theta_j\})$ has \emph{size at least $d$} if for every non-defective ${\leq k}$-configuration $\vec P$ with distinct singletons, there are at least $d$ $k$-tuples in ${W\choose k}$ with $\Theta(e)=\vec P$.

  Observe that, for every $(W,\zeta,\{\theta_j\})$, there is some $d$ so that whenever $(W',\zeta',\{\theta'_j\})$ has size at least $d$, there is an embedding $\pi:W\rightarrow W'$ so that, for all $e$, $\theta_j(e)=\theta'_j(\pi(e))$.  Furthermore, for every $d$, there is an $m$ so that for any $(W,\zeta,\iota,\{\theta_j\})$ with size at least $m$, there is a $W_0\subseteq W$ so that $(W_0,\zeta,\iota,\{\theta_j\})$ has size at least $d$ and is homogeneous.

  So, for each $d$, we can take this large enough $m$ and fix a $(W,\zeta,\{\theta_j\})$ of size at least $m$.  We have $t_{W,\zeta}(\Omega,\xi)>0$ and, for each $\vec x_W\in T_{W,\zeta}(\Omega,\xi)$, we have a $W_0\subseteq W$ of size at least $d$ so that $\vec x_{W_0}$ is homogeneous (that is, there is a $\nu:\mathcal{Z}\rightarrow\Sigma$ so that, for $e\in{W_0\choose k}$, $\mathcal{P}(\vec x_e)=\nu(\Theta(e))$---equivalently, $(W_0,\zeta,\iota,\{\theta_j\})$ is $\nu$-homogeneous, where $\iota$ is induced by $\vec x_{W_0}$).  Since there are finitely many $\nu$, there must be some $\nu$ which we obtain for a set of $\vec x_W$ of positive measure.  Such a $\nu$ exists for every $m$, so there is some $\nu$ which works for arbitrarily large $m$.

  We pick such a $\nu$ and use it to complete the definition of $\rho'$: when $\vec x\in T_{\biguplus_{j'}{[k]\choose j'}}(\vec P)$ for some $\vec P\in\mathcal{Z}$, we set $\rho'(\vec x)=\nu(\vec P)$.

  \textbf{Checking that removal holds}: All that remains is show that whenever $T_{W,c}(\Omega,\rho')\neq\emptyset$ that $t_{W,c}(\Omega,\rho)>0$.

  Suppose we are given a finite $W$ and $(W,c)$ so that $T_{W,c}(\Omega,\rho')\neq\emptyset$.  Choose some $\vec x_W\in T_{W,c}(\Omega,\rho')$.

  From $\vec x_W$ we can read off $(W,\zeta,\{\theta_j\})$ by setting, for $e\in{W\choose r_j}$, $\theta_j(e)=\vec Q^{\{\mathcal{N}_{j'}^{c_{j'}}(\vec x_{e'})\}}_{e'\in\bigcup_{j'\leq j}{e\choose r_{j'}}}$ and $\zeta(e)=\xi(\vec x_e)$.  There is $d$ so that $(W,\zeta,\{\theta_j\})$ will embed in any blow up of the partition of size at least $d$.  Take $(W',\zeta',\{\theta'_j\})$ of size at least $m$ where $m$ is large enough, $t_{W',\zeta'}(\Omega,\rho)>0$ and therefore, by our choice of $\nu$, there is a positive measure of $\vec y_{W'}\in T_{W',\zeta'}(\Omega,\rho)$ such that there is a $W_0\subseteq W'$ so that $(W_0,\zeta',\{\theta'_j\})$ has size at least $d$ and $\vec y_{W_0}$ is $\nu$-homogeneous.

  So consider one of these $\vec y_{W_0}$ where, for all $e\in{W_0\choose k}$, $\vec y_e$ is a point of density for each $f_\sigma$ and a positive point of density for $f_{\mathcal{P}(\vec y_e)}$.  Fix an embedding $\pi:W\rightarrow W_0$.  We claim that, for each $e\in{W\choose k}$, $f_{c(e)}^+(\vec y_{\pi(e)})>0$.

  We consider two cases.  First, suppose $\Theta(e)\not\in\mathcal{Z}$, so $\Theta(e)$ is a non-defective ${\leq}k$-configuration with distinct singletons.  Then $c(e)=\rho'(\vec x_e)$ and, by the definition of $\rho'$, $\rho'(\vec x_e)\in\xi(\vec p^{\Theta(e)})=\zeta(e)=\zeta'(\pi(e))=\xi(\vec y_{\pi(e)})$.  Therefore $c(e)\in\xi(\vec y_{\pi(e)})$, so $f_{c(e)}^+(\vec y_{\pi(e)})>0$.

  Otherwise, we have $\Theta(e)\in\mathcal{Z}$.  Again $c(e)=\rho'(\vec x_e)$ and, by the definition of $\rho'$, $\rho'(\vec x_e)=\nu(\Theta(e))$.  Since $\vec y_{W_0}$ is $\nu$-homogeneous, we have $\rho(\vec y_{\pi(e)})=\nu(\Theta(e))$.

  So we can apply Theorem \ref{thm:hypergraph_counting} to $\vec y_{\pi(W)}$, showing that $t_{W,c}(\Omega,\rho)>0$.
\end{proof}


\section{Ordered Hypergraphs}

The work of the previous section applies, with only minimal changes, to ordered hypergraphs.

\begin{definition}
  When $(\Omega,<)$ is a linearly ordered set and $(W,\prec)$ is a finite linear order, we write $\mathrm{O}^<_{W,\prec}$ for the set of ordered $W$-tuples---that is, the set of tuples $\vec x_W\in \Omega^W$ such that whenever $w\prec w'$, $x_w<x_{w'}$.

  When $(W,\prec)$ is a finite, linearly ordered set, $S$ is a collection of subsets of $W$, and $(\Omega,<)$ is a linearly ordered set, an \emph{ordered $(W,S)$-cylinder intersection set}
 is a set of the form
  \[T_{S,\prec}(\{A_e\}_{e\in S})=T_S(\{A_e\}_{e\in S})\cap \mathrm{O}^<_{W,\prec}.\]
  We define $t_{S,\prec}(\{A_e\}_{e\in S})=\mu(T_{S,\prec}(\{A_e\}_{e\in S})$ and, more generally
  \[t_{S,\prec}(\{f_e\}_{e\in S})=\int\prod_{e\in S}f_e(\vec x_e)\cdot\chi_{\mathrm{O}^<_{W,\prec}}(\vec x_W)\,d\mu.\]

  When $H=(W,\prec,F)$ is a finite ordered $k$-graph and $G=(\Omega,<,E)$ is an ordered $k$-graph, we defined the \emph{ordered induced copies of $H$ in $G$}, written $T^{ind}_{H,\prec}(G)$, to be $T^{ind}_{{W\choose k},\prec}(\{A_e\})$ where $A_e=\left\{\begin{array}{ll}E&\text{if }e\in F\\\Omega^k\setminus E&\text{otherwise}\end{array}\right.$.  We define $t^{ind}_{H,\prec}(G)=\mu(T^{ind}_{H,\prec}(G))$.
\end{definition}

We wish to prove a removal theorem for ordered hypergraphs---that is, when $G=(\Omega,<,E)$ is an ordered hypergraph and $t^{ind}_{H,\prec}(G)=0$, there is an $E'$ with $\mu(E'\bigtriangleup E)<\epsilon$ so that $T^{ind}_{H,\prec}(G')=\emptyset$ when $G'=(\Omega,<,E')$.  (If we were willing to allow $\prec$ itself to be modified, we could use the general removal result from \cite{MR4153699}.  Similarly, if we were only considered with general (i.e. not induced) ordered copies, the result would be immediate: the only issue will be what happens right on the diagonal of $<$, so if we could simply delete a small number of edges near the diagonal, this would follow immediately from ordinary hypergraph removal.)

A crucial fact is that, even though a linear ordering is a binary relation, it is ``explained by'' properties of singletons.  We want to consider a new $\sigma$-algebra.
\begin{definition}
  For each $k$, $\mathcal{B}_{k,<}$ is the the sub-$\sigma$-algebra of $\mathcal{B}_k$ generated by all products $\prod_{i\leq k}I_i$ where each $I_i$ is an interval in $<$.
\end{definition}
Note that, by definition, $\mathcal{B}_{k,<}\subseteq\mathcal{B}_{k,1}$.

\begin{lemma}
  If $\{(\Omega^k,\mathcal{B}_k,\mu_k)\}_{k\in\mathbb{N}}$ then $\{(x,y)\mid x<y\}\in\mathcal{B}_{2,<}$.
\end{lemma}
\begin{proof}
  We show that, for any $\epsilon>0$, we may approximate $\{(x,y)\mid x<y\}$ to within $\epsilon$.  Given $\epsilon>0$, write $\Omega=\bigcup_{i<n}I_i$ where the $I_i$ are disjoint intervals (open or closed) with $\mu_1(I_i)<\epsilon$.   Let $O_\epsilon$ be the union, over all pairs $i<j<n$, of $I_i\times I_j$.  Then $O_\epsilon$ is a product, so belongs to $\mathcal{B}_{2,<}$, $O_\epsilon\subseteq\{(x,y)\mid x<y\}$, and if $x<y$ but $(x,y)\not\in O_\epsilon$, it must be that $(x,y)\in I_i^2$ for some $i<n$.  But the measure of the diagonal $I_i^2$ is less than $\epsilon$.

  Since this holds for every $\epsilon>0$, $\{(x,y)\mid x<y\}\in\mathcal{B}_{2,<}$.
\end{proof}

In general, this means that when $(W,\prec)$ is a partially ordered set, the set of $\vec x_W$ which respect the partial ordering is $\mathcal{B}_{W,<}$-measurable, since it is an intersection of sets of the form $\{x_W\mid x_w<x_{w'}\}$.

\begin{theorem}\label{thm:ordered_hypergraph_counting}
  Let $(W,\prec)$ be a partially ordered finite set, let $\mathcal{N}_d,\ldots,\mathcal{N}_1,\mathcal{N}_{1,<}$ be a properly aligned sequence of systems of neighborhoods so that $\mathcal{N}_i$ is a nested system of neighborhoods with arity $r_i$, let $S$ be a set of subsets of $W$, and suppose that $\vec p_W$ is a $\mathcal{N}_d,\ldots,\mathcal{N}_1,\mathcal{N}_{1,<}$-type such that:
  \begin{itemize}
  \item when $w\prec w'$, there is an $i$ so that $\vec p_w(i)<\vec p_{w'}(i')$,
  \item for each $e\in S$, the restriction $\vec p_W$ is a positive dense type for $f_e$, and
  \item for each $e\in S$, either:
    \begin{itemize}
    \item $f_e$ is $\mathcal{K}_{e,r_d}(\mathcal{N}_d)$-measurable, or
    \item for every $e'\in S\setminus\{e\}$, the function $\vec x_e\mapsto \int f_{e'}(\vec x_{e\cup e'})\,d\mu(\vec x_{e'\setminus e})$ is $\mathcal{K}_{e,r_d}(\mathcal{N}_d)$-measurable.
    \end{itemize}
  \end{itemize}
  Then $t_{S,\prec}(\{f_e\})>0$.
\end{theorem}
\begin{proof}
  We proceed by induction on $d$.  When $d=0$ (that is, we are considering a $\mathcal{N}_{1,<}$-type), the proof is similar to Theorem \ref{thm:graph_counting}, taking care to respect the ordering.
  
Choose some $\epsilon\leq\min_{e\in S}f^+_e(\vec x_e)$.

Since each $\vec p_e$ is a dense type, we may choose some $j$ large enough that, for each $e\in S$, 
\[\frac{1}{\mu(\vec p_e(j))}\mu(\{\vec y_e\in \vec p_e(j)\mid f^+_e(\vec y_e)\geq\epsilon/2\})\geq 1-\frac{1}{|S|}.\]
Consider $\prod_{w\in W}p_w(j)$.  This is a product of intervals and, when $j$ is large enough, the map $w\mapsto p_w(j)$ is order preserving.  Therefore 
\[\frac{1}{\mu(\vec p_W(j))}\mu(\{\vec y_W\in \vec p_W(j)\mid f^+_e(\vec y_e)\geq\epsilon/2\text{ and whenever }w\prec w'\text{, }y_w<y_{w'}\})\geq 1-\frac{1}{|S|}.\]

Therefore
\[\gamma=\mu(\{\vec y_W\in\vec p_W(j)\mid \text{ there is some }e\in S\text{ such that }f^+_e(\vec y_e)<\epsilon/2\})<|S|\frac{1}{|S|}=1,\]
so, using Lemma \ref{thm:reduction},
\[t_{S,\prec}(\{f_e\})=t_{S,\prec}(\{f^+_e\})\geq\frac{\epsilon^{|S|}}{2^{|S|}}(1-\gamma)>0.\]

The argument from Theorem \ref{thm:hypergraph_counting} applies unchanged for the inductive case since the set of ordered tuples is $\mathcal{N}_i$-measurable for all $i$.
\end{proof}

\begin{theorem}[Ordered Hypergraph Removal]\label{thm:ordered_hypergraph_removal}
  Let $\Sigma$ be a finite set and let $(\Omega,\rho,<)$ be given along with a countably approximated atomless Keisler graded probability space on $\Omega$ with $\rho:{\Omega\choose k}\rightarrow\Sigma$ such that each $\rho^{-1}(\sigma)\in\mathcal{B}_2^0$ and a dense collection of intervals of $<$ is in $\mathcal{B}_1^0$.  For each $\epsilon>0$ there is a $\rho':{\Omega\choose k}\rightarrow\Sigma$ such that
  \[\mu(\{\vec x\in\Omega^k\mid\rho(\vec x)\neq \rho'(\vec x)\})<\epsilon,\]
each $(\rho')^{-1}(\sigma)\in\mathcal{B}_{2,0}$, and, for all $(W,c,\prec)$, if $t_{W,c,\prec}(\Omega,\rho)=0$ then $T_{W,c,\prec}(\Omega,\rho')=\emptyset$.
\end{theorem}
\begin{proof}
  The proof is largely unchanged from the proof of Theorem \ref{thm:induced_hypergraph_removal} using the sequence of systems of neighborhoods $\mathcal{N}_{k-1},\mathcal{N}_{k-2},\ldots,\mathcal{N}_1,\mathcal{N}_{1,<}$, so we take $d=k+1$, and using Theorem \ref{thm:ordered_hypergraph_counting}.

  The only further step that needs to be checked carefully is the homogenization step when dealing with tuples with distinct singeltons.  Our definition of a blow up $(W,\zeta,\{\theta_j\})$ is unchanged---note that a partial ordering of $W$, on pairs where $\theta_1$ is injective, can be inferred from the assignment $\theta_1$.  Our homogeneous blowups $(W,\zeta,\iota,\{\theta_j\},\prec)$ are defined to have total orderings where $\prec$ is consistent with $\theta_1$.  The crucial point is that the Ramsey-type property still holds: for every $d$, there is an $m$ so that for any $(W,\zeta,\iota,\{\theta_j\},\prec)$, there is a $W_0\subseteq W$ so that $(W_0,\zeta,\iota,\{\theta_j\},\prec)$ is homogeneous. The ordering is no obstacle to obtaining this by the usual Ramsey theoretic arguments, and the rest of the proof is unchanged.
\end{proof}

\begin{cor}[Ordered hypergraph removal lemma]\label{thm:main_ordered}
  For every finite set $\Sigma$, every $\epsilon>0$, and every $\Sigma$-colored $(W,\prec,c)$, there is a $\delta>0$ so that for any $(\Omega,<,\rho)$ with $t_{W,\prec,c}(\Omega,<,\rho)<\delta$ there is a $\rho'$ with
  \[|\{\vec x\in \Omega^k\mid \rho(\vec x)\neq\rho'(\vec x)\}|<\epsilon|\Omega|^k\]
such that $T_{W,\prec,c}(\Omega,<,\rho)=\emptyset$.
\end{cor}

\begin{cor}[Infinite ordered hypergraph removal lemma]
  For every finite set $\Sigma$, every $\epsilon>0$, and every family $\mathcal{F}$ of finite $\Sigma$-colored ordered hypergraphs, there are $\delta>0$ and a bound $M$ so that for any $(\Omega,<,\rho)$, if, for every $(W,\prec,c)\in\mathcal{F}$ with $|W|\leq M$ we have $t_{W,\prec,c}(\Omega,<,\rho)<\delta$, then 
  there is a $\rho'$ with
  \[|\{\vec x\in \Omega^k\mid \rho(\vec x)\neq\rho'(\vec x)\}|<\epsilon|\Omega|^k\]
such that for every $(W,\prec,c)\in\mathcal{F}$, $T_{W,\prec,c}(\Omega,<,\rho)=\emptyset$.
\end{cor}

\section{Further Directions}

We have not attempted to identify the correct common generalization of Theorems \ref{thm:induced_hypergraph_removal} and \ref{thm:ordered_hypergraph_removal} to give a general theorem saying that certain structures can be removed while preserving some fixed structure.  Such a theorem must make some promise about the measurability of the fixed structure, and additionally place some sort of Ramsey-type condition on it.

There are other examples in the literature where some distinguished family of sets analogous to $\mathcal{B}_<$ is of particular interest.  In particular, \cite{MR4136157} considers a computational setting; translated into our framework here, we add the assumption that the points of $\Omega$ are understood to have a structure like binary sequences $2^\Lambda$, embodied in a distinguished family of sets $\mathcal{B}_{1,c}$ which consists of those sets
\[[s]=\{\omega\mid \forall \lambda\in\dom(s)\,s(\lambda)=\omega(\lambda)\]
where $s$ is a partial function with finite domain from $\Lambda$ to $\{0,1\}$.  That is, the distinguished sets are those in which a finite number of coordinates have been fixed.  The regularity lemma they prove is precisely the one corresponding to the sequence of $\sigma$-algebras $\mathcal{B}_1,\mathcal{B}_{1,c}$; extending the removal lemma to this setting (or to longer sequences $\mathcal{B}_d,\ldots,\mathcal{B}_1,\mathcal{B}_{1,c}$) would require identifying interesting structures to be the fixed part (analogous to the ordering) which are $\mathcal{B}_{1,c}$-measurable---that is, the relation symbols in this structure would have to have the property that they can be calculated on all but measure $\epsilon$ points while examining the input at only finitely many points in $\Lambda$.

\printbibliography
\end{document}